\documentclass{article}
\usepackage{amsmath}
\usepackage{amssymb}
\usepackage[utf8x]{inputenc}
\usepackage{graphicx}
\usepackage{dsfont}
\usepackage{hyperref}
\usepackage{amsthm}
\usepackage{tikz}
\usepackage{tikz-3dplot}
\usepackage{./gnuplot-lua-tikz}
\usepackage{subcaption}
\usepackage{bm}
\usepackage{authblk}

\newcommand{\B}{\mathcal{B}}
\newcommand{\Q}{\mathcal{Q}}
\renewcommand{\L}{\mathcal{L}}
\renewcommand{\S}{\mathcal{S}}

\renewcommand{\vector}[1]{\begin{pmatrix}#1\end{pmatrix}}
\renewcommand{\O}{\mathcal{O}}
\newcommand{\R}{\mathbb{R}}

\newcommand{\ub}[1]{\underline{#1}}
\newcommand{\etal}{\textit{et al.\ }}
\newcommand{\AS}[1]{\mathfrak{#1}}
\renewcommand{\vec}[1]{\bm{#1}}

\renewcommand{\d}{\mathrm{d}}
\newcommand*\centermathcell[1]{\omit\hfil$\displaystyle#1$\hfil\ignorespaces}

\providecommand{\keywords}[1]{\textbf{\textit{Keywords:}} #1}
\providecommand{\license}{\textbf{\textit{License:}} CC-BY-NC-ND\footnote{https://creativecommons.org/licenses/by-nc-nd/3.0/}}

\DeclareMathOperator{\diag}{diag}

\newtheorem{theorem}{Theorem}

\newtheorem{corollary}{Corollary}

\colorlet{gp}{black}

\begin{document}

\title{Derivation and Analysis of Lattice Boltzmann Schemes for the Linearized Euler Equations}
\author[1,2]{Philipp Otte\thanks{otte@mathcces.rwth-aachen.de}}
\author[1]{Martin Frank\thanks{frank@mathcces.rwth-aachen.de}}
\affil[1]{Department of Mathematics, RWTH Aachen University, Schinkelstr. 2, 52062 Aachen, Germany}
\affil[2]{German Research School for Simulation Sciences GmbH, Schinkelstr. 2a, 52062 Aachen, Germany}
\date{21.12.2015}

\maketitle

\begin{abstract}
We derive Lattice Boltzmann (LBM) schemes to solve the Linearized Euler Equations in 1D, 2D, and 3D with the future goal of coupling them to an LBM scheme for Navier Stokes Equations and an Finite Volume scheme for Linearized Euler Equations.
The derivation uses the analytical Maxwellian in a BGK model.
In this way, we are able to obtain second-order schemes.
In addition, we perform an $L^2$-stability analysis.
Numerical results validate the approach.
\end{abstract}

{\setlength{\parindent}{0cm}\keywords{Lattice Boltzmann Method, Finite Discrete Velocity Models, Linearized Euler Equations, Asymptotic Analysis, Stability}
\medskip
\\
\license}

\section{Introduction}
\label{sec: Introduction}

In the field of computational aeroacoustics (CAA) techniques from computational fluid dynamics are used to predict aeroacoustic phenomena.
In a project with an industry partner, we are investigating the aeroacoustic far-field generated by highly vortical flows streaming through a flat plate silencer built from porous media.
Hasert \cite{DissManuel} identified three different length-scales within this setting: 1.\ the size of pores within the porous medium ($\O(10 \mu m)$); 2.\ the length of the vortical flow ($\O(mm)$); and 3.\ the dimension of the inviscid acoustic far-field ($\O(m)$).
In order to resolve the acoustic effects of the porous medium properly, the Lattice Boltzmann method (LBM) is used to solve the Navier Stokes Equation (NSE).
Since within the acoustic far-field viscosity can be neglected, we simulate the far-field using the Linearized Euler Equation (LEE).
We plan to use Finite Volume methods (FVM) for solving the LEE, due to their conservation form and the large length-scale of the acoustic far-field.
The latter point is important for keeping computational complexity manageable.
In order to follow this approach, a coupling of the kinetic LBM for the NSE and the macroscopic FVM for the LEE is necessary.
\medskip
\\
This coupling itself introduces two difficulties.
First, we nee to couple the viscous NSE with the inviscid LEE.
Here, the problem is that we expect an abrupt change between those models to introduce non-physical effects such as spurious reflections of sound waves.
We plan to use a smooth transition model similar to \cite{DegondJin2005} and \cite{Degond2005} smearing the change in viscosity over a buffer zone in order to minimize these non-physical effects. 
Second, we need to couple the mesoscopic LBM with the macroscopic FVM.
Here, the translation from mesoscopic particle density functions to macroscopic quantities can be done using moments while the other way round is non-trivial.
Again, we are especially interested in avoiding non-physical effects such as the above mentioned spurious reflections of sound waves.
In order to be able to handle both problems separately, we decided to split this coupling of both model and method into two steps: 1.\ switch the model from NSE to LEE; and 2.\ switch the method from LBM to FVM.
This provides us with the possibility to separate the derivation of a smooth transition model for the change from NSE to LEE  from the choice of appropriate translations of macroscopic to mesoscopic quantities.
By this, we gain the freedom to choose approaches ideal for each problem.
The drawback of this two-step approach is that it introduces a buffer-zone in which we need to solve the LEE using LBM.
In this paper we present an intermediate step in the derivation of the needed LBM of the LEE used in the aforementioned buffer-zone: LBM schemes for the LEE without background velocity.
Research on how to extend these LBM schemes to incorporate background velocities is already undergoing.
\medskip
\\
During the last 20 years, lots of research on the LBM was undertaken.
He and Luo \cite{HeLuo1997} showed that the LBM is not only a generalization of Lattice Gas Automata (LGA) but is a discretization of the Boltzmann Bhatnagar-Gross-Krook (BGK) equation.
In \cite{Junk2005} Junk \etal performed a rigorous analysis of the LBM using well-established tools from analysis of Finite Difference Methods (FDM).
For the problem of stability of LBM, different approaches were investigated: 1. direct von Neumann analysis of the linearized LBM (Sterling and Chen \cite{SterlingChen96}, Lallemand and Luo \cite{Lallemand2000a}); 2. entropic LBM with equilibrium distributions that admit an \textit{H}-theorem (Chen and Teixeira \cite{HTheoremInstability}, Karlin \etal \cite{EntropyFunctionsLBM}); 3. rigorous stability analysis with respect to a weighted $L^2$-norm (Banda \etal \cite{banda2006}, Junk and Yong \cite{JunkYong2009}, Junk and Yang \cite{JunkYang2009}); and 4. application of concepts from nonequilibrium thermodynamics (Yong \cite{Yong2009862}).
Bernsdorf \etal \cite{bernsdorf1999comparison} first showed the suitability of the LBM for flows through complex geometries.
In addition, Buick \etal \cite{buick1998lattice}, Dellar \cite{dellar2001bulk}, Crouse \etal \cite{crouse2006fundamental}, Lallemand and Luo \cite{Lallemand2000a} \cite{Lallemand2003}, and Marie \etal \cite{Marie2009} analyzed, tested, and discussed the LBM as tool for acoustic simulations.
Based on the capability of the LBM for complex geometries and acoustics, Hasert \etal \cite{hasert2011towards} and Hasert \cite{DissManuel} used the LBM enhanced by a sub-grid model to resolve the acoustic field generated by a flow through a porous medium.
By using the LBM, they were able to directly simulate the aeroacoustic contributions of the individual pores.
Independent of this, in the field of continuous analysis of the Boltzmann equation Bardos \etal \cite{Bardos2000} derived the Linearized Euler equations for monoatomic gases as limit of the continuous Boltzmann equation in acoustic scaling.
This paper was the starting point of the derivations presented in this paper.
The use of LEE for CAA offers a resource-saving alternative over the classical use of NSE due to their simpler, linear structure.
Mankbadi \etal \cite{mankbadi1998use} used the LEE for simulation of supersonic jet noise.
Further studies of the acoustic capabilities of the LEE were for example taken out by Bailly and  Juvé \cite{bailly2000numerical} and Bogey \etal \cite{bogey2002computation}.
Roller \etal \cite{roller2005} showed that in a hybrid approach the LEE are a well-suited and efficient method for simulating the acoustic far-field.
\medskip
\\
This work is structured as follows.
First, in section \ref{sec: Theoretical Background} we briefly present the Linearized Euler Equations without background velocity as used in this paper and briefly recapitulate the basics of the Boltzmann Equation.
In section \ref{sec: lee as limit of boltzmann equation}, we adapt the results by Bardos \etal \cite{Bardos2000} to the LEE for monoatomic gases as used in this paper.
Based on this, we then derive the semi-discrete Finite Discrete Velocity Models for monoatomic gases in section \ref{sec: FDVM} and generalize these results to polyatomic gases in section \ref{sec: FDVM Poly}.
In these sections, we will also post necessary conditions on the velocity models.
In sections \ref{sec: LBE} and \ref{sec: Analysis of the LBE}, we then present the fully discrete Lattice Boltzmann Equation and perform an analysis of consistency and stability of this equation.
Based on the necessary conditions derived in sections \ref{sec: FDVM} and \ref{sec: FDVM Poly}, in section \ref{sec: examples of velocity sets} we then present velocity models for monoatomic and diatomic gases respectively.
In addition, stability of the LBM for these velocity models is analyzed using the results from section \ref{sec: Analysis of the LBE}.
Then, numerical results are presented and discussed in section \ref{sec: Numerical Results}.
Finally, in section \ref{sec: conclusions} we wrap up the results of this paper.

\section{Theoretical Background}
\label{sec: Theoretical Background}

\subsection{The Linearized Euler Equations}
\label{sec: LEE}

In the field of Fluid Dynamics the compressible Euler Equations describe inviscid flows.
In settings in which the fluid flow is dominated by a constant background flow the Euler Equations can be linearized around this flow to reduce their complexity.
Assume a constant background flow with density $\rho_0$, temperature $\theta_0$, and velocity $\vec{u}_0$.
Linearization of the Euler Equations around this background flow then yields the Linearized Euler Equations (LEE).
In this setting the macroscopic variables of the fluid density $\rho$, velocity $\vec{u}$, temperature $\theta$, and pressure $p$ are given as:
\begin{subequations}
\begin{align}
\rho &= \rho_0 + \epsilon \rho',\\
u &= \vec{u}_0 + \epsilon \vec{u}',\\
\theta &= \theta_0 + \epsilon \theta',\\
p &= p_0 + \epsilon p'.
\end{align}
\end{subequations}
where the primed variables represent fluctuations around the background flow and the parameter $\epsilon$ represents the scale of the fluctuations.
The velocities $\vec{u}$ and $\vec{u}'$ are $\vec{u} = u_x$ and $\vec{u}' = \vec{u}'_x$ respectively in the 1D case, $\vec{u} = ( u_x, u_y )^T$ and $\vec{u}' = ( u'_x, u'_y )^T$ respectively in the 2D case, and $\vec{u} = ( u_x, u_y, u_z )^T$ and $\vec{u}' = ( u'_x, u'_y, u'_z )^T$ respectively in the 3D case. 
The temperature $\theta$ used throughout this paper does not describe the temperature $T$ in Kelvin but the scaled temperature $\theta = T R_{specific}$ where $R_{specific}$ denotes the specific gas constant.
Since the pressure in an ideal gas is given by $p = \rho \theta$, for the fluctuations in pressure we have $p' = \rho_0 \theta' + \theta_0 \rho'$.
\medskip
\\
As already stated in the introduction, we focus on the LEE for flows without background velocity, i.e. $\vec{u}_0 = \vec{0}$.
This can either be the case if no background velocity is present, its magnitude is small enough to incorporate it into the fluctuations $\vec{u}'$, or by choice of an appropriate Galilean frame.
The assumption of a such a Galilean frame can be easily justified for the continuous case.
Under certain conditions, one might be able to implement such a Galilean frame for boundary-free problems using a moving lattice.
The dimensional LEE without background velocity are given as follows:
\begin{subequations}
\begin{align}
\partial_{\ub{t}} \ub{\rho}' + \ub{\rho}_0 \nabla_{\ub{\vec{x}}} \cdot \ub{\vec{u}}' &= 0,\\
\ub{\rho}_0 \partial_{\ub{t}} \ub{\vec{u}}' + \ub{\rho}_0 \nabla_{\ub{\vec{x}}} \ub{\theta}' + \ub{\theta}_0 \nabla_{\ub{\vec{x}}} \ub{\rho}' &= \vec{0},\\
\frac{1}{\gamma - 1} \ub{\rho}_0 \partial_{\ub{t}} \ub{\theta}' + \ub{\rho}_0 \ub{\theta}_0 \nabla_{\ub{\vec{x}}} \cdot \ub{\vec{u}}' &= 0.
\end{align}
\end{subequations}
Here, $\gamma$ denotes the adiabatic exponent of the gas which will be discussed in section \ref{sec: FDVM Poly}.
To obtain the non-dimensional variables $\rho_0$, $\rho'$, $\vec{u}'$, $\theta_0$, $\theta'$, $x$, and $t$ from the dimensional quantities $\ub{\rho_0}$, $\ub{\rho'}$, $\ub{\vec{u}'}$, $\ub{\theta_0}$, $\ub{\theta'}$, $\ub{x}$, and $\ub{t}$ the following conversions are used:
\begin{align*}
\ub{\rho_0} &= \rho_0^* \rho_0, & \ub{\rho'} &= \rho_0^* \rho', & \ub{\theta_0} &= \theta_0^* \theta_0,\\
\ub{\theta'} &= \theta_0^* \theta', & \ub{\vec{u}'} &= u^* \vec{u}', & u^* &= \frac{x^*}{t^*},\\
(u^*)^2 &= \theta_0^*, & \ub{x} &= x^* x, & \ub{t} &= t^* t,
\end{align*}
where $u^*$ is scalar.
The non-dimensional LEE are:
\begin{subequations}
\label{LEE}
\begin{align}
\partial_t \rho' + \rho_0 \nabla_{\vec{x}} \cdot \vec{u}' &= 0,\\
\rho_0 \partial_t \vec{u}' + \rho_0 \nabla_{\vec{x}} \theta' + \theta_0 \nabla_{\vec{x}} \rho' &= \vec{0},\\
\frac{1}{\gamma - 1} \rho_0 \partial_t \theta' + \rho_0 \theta_0 \nabla_{\vec{x}} \cdot \vec{u}' &= 0.
\end{align}
\end{subequations}
As important outcome of this dedimensionalization, we are free to choose the non-dimensional background density $\rho_0$ and the temperature $\theta_0$ independently of each other.
Therefore, a numerical scheme for non-dimensional flows with background density $\rho_0$ and temperature $\theta_0$ can be used to simulate dimensional flows with arbitrary dimensional background densities $\ub{\rho}_0$ and temperature $\ub{\theta}_0$.
In sections \ref{sec: examples of velocity sets} and \ref{sec: Numerical Results} we will make heavy use of this important property of the LEE without background velocity.

\subsection{The Boltzmann Equation}
\label{sec: Boltzmann Equation}

The Boltzmann equation describes fluids on the mesoscopic kinetic level.
This means, it does neither model the interactions of distinct particles (microscopic level) nor the evolution of macroscopic properties (macroscopic level), such as density or pressure, but models the probability of particles at certain position and velocity.
This is done via particle density functions defined as follows:
$$
F: \mathbb{T} \times \mathbb{X} \times \mathbb{V} \to \mathbb{R}_0^+.
$$
Here, $\mathbb{X} \subseteq \mathbb{R}^D$ denotes the spatial domain, $\mathbb{V} = \mathbb{R}^D$ the velocity-space, $\mathbb{T} \subseteq \mathbb{R}_0^+$ the time-frame of interest, and $D$ the dimension of the problem.
In this paper we restrict the spatial domain $\mathbb{X}$ to be periodic.
The classical Boltzmann equation is now given by:
\begin{equation}
\label{eq: cont Boltzmann}
\partial_t F + \vec{v} \cdot \nabla_{\vec{x}} F = \B(F,F)
\end{equation}
with initial data:
$$
F(0,\vec{x},\vec{v}) = F_0(\vec{x},\vec{v}) \geq 0
$$
and the Boltzmann collision operator $\B$.
To obtain the macroscopic variables of the gas we calculate the moments of the particle density function $F$.
These moments are defined as follows:
$$
\langle \zeta F \rangle = \int_\mathbb{V} \zeta(\vec{v}) F(\vec{v}) \d\vec{v}
$$
For further details please refer to \cite{Cercignani1988} and \cite{Bardos2000}.
It is important to note, that the classical Boltzmann equation is of dimension $2D+1$ and via the collision operator the equation is coupled for all velocities $\vec{v} \in \mathbb{V}$.
These two properties make straight forward strategies for solving the Boltzmann equation extremely expensive.

\section{The LEE for Monoatomic Gases as Limit of the Boltzmann Equation}
\label{sec: lee as limit of boltzmann equation}

Since the Boltzmann equation models monoatomic particles, only the LEE for monoatomic gases can easily be derived as a limit of the Boltzmann equation.
Bardos \etal in \cite{Bardos2000} proved that the LEE for monoatomic gases with $\rho_0 = \theta_0 = 1$ are a limit of the Boltzmann equation in acoustic scaling.
In this section, we go through this derivation and adapt it to arbitrary $\rho_0,\theta_0 >0$.
This is important as it allows for additional degrees of freedom in the derivation of finite discrete velocity models as presented in sections \ref{sec: FDVM}, \ref{sec: FDVM Poly}, and \ref{sec: examples of velocity sets}.
For further information on the derivation process refer to the original paper by Bardos \etal \cite{Bardos2000}.
\medskip
\\
The Boltzmann Equation in the acoustic scaling is given by:
\begin{equation}
\label{Eq: Cont Boltzmann acoustic scaling}
\partial_t F + \vec{v} \cdot \nabla_{\vec{x}} F = \frac{1}{\epsilon} \B(F,F),
\end{equation}
where $\epsilon$ denotes the Knudsen number.
As the LEE describe fluctuations around a background stream we transfer this approach to the Boltzmann Equation.
Thus, we analyze fluctuations around the spatially invariant background equilibrium
$$
M(\vec{v}) = \frac{\rho_0}{(2 \pi \theta_0)^{\frac{D}{2}}} \exp(-\frac{1}{2 \theta_0} |\vec{v}|^2).
$$
The moments of $M$ resemble the background flow with density $\rho_0$, temperature $\theta_0$, and velocity $\vec{u}_0=\vec{0}$.
We linearize the particle density function $F$:
$$
F_\epsilon = M G_\epsilon = M ( 1 + \epsilon g_\epsilon ),
$$
where $\epsilon M g_\epsilon$ represent the fluctuations around the background flow.
The macroscopic fluctuations are given by the following moments:
\begin{subequations}
\begin{align}
\rho' &= \langle g_\epsilon M \rangle,\\
\vec{u}' &= \frac{1}{\rho_0} \langle \vec{v} g_\epsilon M \rangle,\\
\theta' &= \frac{1}{\rho_0} \left( \frac{1}{D} \langle |\vec{v}|^2 g_\epsilon M \rangle - \theta_0 \langle g_\epsilon M \rangle \right).
\end{align}
\end{subequations}
With this linearization equation (\ref{Eq: Cont Boltzmann acoustic scaling}) can be rewritten as:
\begin{equation}
\label{Eq: Cont Boltzmann acoustic scaling fluctuations only}
\partial_t g_\epsilon + \vec{v} \cdot \nabla_{\vec{x}} g_\epsilon + \frac{1}{\epsilon} \L g_\epsilon = \Q(g_\epsilon,g_\epsilon),
\end{equation}
with an adapted collision operator $\Q$ and its linearized version $\L$.
Analogous to \cite{Bardos2000} we find that for $\epsilon \to 0$ the fluctuation $g$ takes the form (also cf. \cite{Cercignani1988}):
$$
g = \AS{a} + \vec{v} \cdot \AS{b} + \frac{1}{2} |\vec{v}|^2 \AS{c}.
$$
By comparison of the moments regarding $\left\{ M, M v_1,..., M v_D, \frac{1}{2} M |\vec{v}|^2\right\} $ of equation (\ref{Eq: Cont Boltzmann acoustic scaling fluctuations only}) with the LEE (\ref{LEE}) we derive the form:
\begin{equation}
\label{limit LEE}
g = \frac{1}{\rho_0} \rho' + \frac{1}{\theta_0} \vec{v} \cdot \vec{u}' + \theta' \left( \frac{1}{2\theta_0^2} |\vec{v}|^2 - \frac{D}{2\theta_0} \right).
\end{equation}
In the next section, we will use this Maxwellian as equilibrium distribution in the derivation of a linearized semi-discrete version of equation (\ref{Eq: Cont Boltzmann acoustic scaling}) for monoatomic gases.
Throughout this paper, we will refer to these semi-discrete models as Finite Discrete Velocity Models (FDVM) though in literature also the names discrete Boltzmann equation \cite{CGLBM}, differential form of the Lattice Boltzmann Equation \cite{SucciBook}, and Lattice Boltzmann Equation \cite{FELBM} are used.
We will only refer to the linearized fully-discrete version of equation (\ref{Eq: Cont Boltzmann acoustic scaling}) derived in section \ref{sec: LBE} as Lattice Boltzmann Equation (LBE).
We would like to stress, that this is the exact Maxwellian resembling the LEE.
This is a difference to the BGK models used in LBM for the NSE where not the exact Maxwellian is used as equilibrium distribution but for example a Taylor expansion of the exact Boltzmann-Maxwellian distribution (cf. \cite{HeLuo1997}).
Here, we avoid errors introduced by this expansion by using the exact Maxwellian as equilibrium in the BGK operator.

\section{Derivation of Finite Discrete Velocity Models for Monoatomic Gases}
\label{sec: FDVM}
In order to discretize eqaution (\ref{Eq: Cont Boltzmann acoustic scaling}) in velocity space, we follow the standard two-step approach.
First, we replace the Boltzmann collision operator $\B(F,F)$ by a BGK type collision opperator.
In BGK (Bhatnagar–Gross–Krook) type collision operators $J_{BGK}(g) = \frac{1}{\tau} \left( g^{eq}(g) - g \right)$ the particle density function $g$ is relaxed towards an equilibrium $g^{eq}$ scaled with the relaxation time $\tau$ (cf. Bhatnagar \etal \cite{BGK1945}).
Therefore, the structure of BGK type collision operators is a lot simpler than the structure of complex Boltzmann type collision operators.
Second, we restrict the infinite velocity space $\mathbb{V}$ to a finite and symmetric set of discrete velocities $\S = \left\{ \vec{c}_i \in \R^D: i=1..n\right\}$.
For a quantity $\zeta : \mathcal{S} \to \mathbb{R}$ the discrete moments are defined as follows:
$$
\langle \zeta \rangle_\mathcal{S} = \sum_{i=1}^n \zeta( \vec{c}_i ),
$$
where the subscript $\mathcal{S}$ indicates the dependence on the finite set of discrete velocities $\mathcal{S}$.
For the sake of simplicity, we will use the notation $\langle \vec{v} \zeta \rangle_{\mathcal{S}} = \sum_{i=1}^n c_i \zeta_i$ and so on in cases where $\zeta$ is a function depending on the velocities.
With this, for each discrete velocity $\vec{c}_i \in \S$ we have the Finite Discrete Velocity Model (FDVM) equation:
\begin{equation}
\label{eq: FDVM equation}
\partial_t g_i + \vec{c}_i \cdot \nabla_{\vec{x}} g_i = \frac{1}{\epsilon \tau} \left( g_i^{eq} - g_i \right),
\end{equation}
where $g_i$ denotes the particle density function corresponding to velocity $\vec{c}_i$, $g_i^{eq}$ denotes an equilibrium function corresponding to velocity $\vec{c}_i$, and $\tau$ denotes the relaxation time introduced by the BGK approximation.
We keep both relaxation parameters $\epsilon$ and $\tau$ as we will use the first as step size in the LBM discretization (cf. section \ref{sec: LBE}) and the latter as parameter allowing for second-order consistency of the LBM schemes (cf. section \ref{sec: Analysis of the LBE}).
\medskip
\\
This discretization leaves the following choices open:
\begin{enumerate}
\item the finite set of discrete velocities $\S$;
\item the relaxation time $\tau$; and
\item the equilibrium distributions $g_i^{eq}$.
\end{enumerate} 
The choices of $\S$ and the equilibrium distributions $g_i^{eq}$ are coupled.
For the equilibrium distributions $g_i^{eq}$ we use the limit derived in the fully continuous case in section \ref{sec: lee as limit of boltzmann equation}.
So for $\vec{c}_i \in \S$ we have:
\begin{equation}
\label{eq: equilirium}
g_i^{eq} = g_i^{eq}(\rho',\vec{u}',\theta') = \left( \frac{1}{\rho_0} \rho' + \frac{1}{\theta_0} \vec{c}_i \cdot \vec{u}' + \theta' \left( \frac{1}{2\theta_0^2} |\vec{c}_i|^2 - \frac{D}{2\theta_0} \right) \right) f^*_i.
\end{equation}
The function $f^*$ with $f^*_i = f^*(\vec{c}_i)$ was introduced to replace the continuous Maxwellian $M$ in the calculation of moments.
This function can be introduced into the equilibrium and thereby into the particle density functions $g_i$.
This needs to reflect in the initial and boundary conditions.
Due to its construction, the function $f^*$ has to be even, symmetric, and positive.
In addition, the following similarity condition has to hold for certain pairs of discrete moments of $f^*$ and continuous moments of $M$:
$$
\langle \zeta f^* \rangle_\mathcal{S} =  \langle  \zeta M \rangle  \hbox{ for } \zeta \in \left\{ \vec{v}\mapsto 1, \vec{v}\mapsto \frac{1}{2} |\vec{v}|^2, \vec{v}\mapsto v_\alpha v_\beta,  \vec{v}\mapsto \frac{1}{2} |\vec{v}|^2 v_\alpha v_\beta, \vec{v}\mapsto \frac{1}{4}|\vec{v}|^4 \right\}.
$$
Hence, only those finite sets of discrete velocities $\S$ fulfilling the conditions above and for which such a function $f^*$ can be found may be considered.
For every pair of a finite set of discrete velocities $\S$ and a function $f^*$ fulfilling the conditions above we get the following equality of the discrete and the continuous moments:
\begin{subequations}
\begin{align}
\rho' &= \langle g \rangle_\S = \langle g_\epsilon M \rangle,\\
\vec{u}' &= \frac{1}{\rho_0} \langle \vec{v} g \rangle_\S = \frac{1}{\rho_0} \langle \vec{v} g_\epsilon M \rangle,\\
\theta' &= \frac{1}{\rho_0} \left( \frac{1}{D} \langle |\vec{v}|^2 g \rangle_\S - \theta_0 \langle g \rangle_\S \right) = \frac{1}{\rho_0} \left( \frac{1}{D} \langle |\vec{v}|^2 g_\epsilon M \rangle - \theta_0 \langle g_\epsilon M \rangle \right).
\end{align}
\end{subequations}
Therefore, the choices of the velocity set $\S$ and the equilibrium distributions $g_i^{eq}$ are coupled.
The choice of the relaxation time $\tau$ is delayed until section \ref{sec: Analysis of the LBE}, where we will fix $\tau = \frac{1}{2}$ in order to achieve second-order consistency.
\medskip
\\
It is important to understand that equation (\ref{eq: FDVM equation}) is coupled for $i=1..n$ via the equilibrium term.
Therefore, solving these equations independently is not possible.

\section{Generalization of the Finite Discrete Velocity Ansatz for Polyatomic Gases}
\label{sec: FDVM Poly}

The derivation process shown above describes monoatomic gases, i.e.\ noble gases, only.
In this section we will generalize the derivation process of FDVM for polyatomic gases.
The LEE without background velocity for polyatomic gases are of the following form:
\begin{align*}
\partial_t \rho' + \rho_0 \nabla_{\vec{x}} \cdot \vec{u}' &= 0,\\
\rho_0 \partial_t \vec{u}' + \rho_0 \nabla_{\vec{x}} \theta' + \theta_0 \nabla_{\vec{x}} \rho' &=\vec{0},\\
\frac{1}{\gamma-1} \rho_0 \partial_t \theta' + \rho_0 \theta_0 \nabla_{\vec{x}} \cdot \vec{u}' &= 0.
\end{align*}
For calorically ideal gases the adiabatic exponent $\gamma$ is given by the relation $\gamma = \frac{M+2}{M}$ where $M$ denotes the number of degrees of freedom -- translational, rotational, and vibrational -- of the gas considered (cf. \cite{toro1999riemann}).
These additional degrees of freedom lead to an additional internal energy not present in the monoatomic case.
A monoatomic gas in $D$ dimensions only has $D$ translational degrees of freedom while a diatomic gas has $D$ translational and $D-1$ rotational degrees of freedom, resulting in $M=2D-1$.
Since 98\%-99\% of ambient air are made up by nitrogen and oxygen molecules we can model ambient air as diatomic gas.
\medskip
\\
When moving from monoatomic to polyatomic gases we have to account for the additional internal energy in the molecule.
We write the energy equation in terms of the energy density $\mathcal{E}$ and the energy flux $\vec{\mathcal{F}}$:
\begin{align*}
\partial_t \mathcal{E} + \nabla_{\vec{x}} \cdot \vec{\mathcal{F}} &= 0.
\end{align*}
Among others, Kataoka and Tsutahara \cite{kataoka2004} and Dellar \cite{dellar2008two} proposed to amend the energy density $\mathcal{E}$ and the energy flux $\vec{\mathcal{F}}$ using additional energies $\beta_i$:
\begin{alignat*}{3}
\mathcal{E} &=& \centermathcell{\frac{1}{\gamma -1} \left( \rho_0 \theta' + \theta_0 \rho' \right)} =& \langle \frac{1}{2} \left( |\vec{v}|^2 + \beta \right) g \rangle_\mathcal{S},\\
\vec{\mathcal{F}} &=&\centermathcell{\frac{\gamma}{\gamma-1} \rho_0 \theta_0 \vec{u}' } =&  \langle \frac{1}{2} \left( |\vec{v}|^2 + \beta \right) \vec{v} g \rangle_\mathcal{S},
\end{alignat*}
where the function $\beta$ of the velocities is predefined, positive, and fulfills the same symmetries as $f^*$.
With this we can now try to construct an equilibrium distribution for which we resemble the LEE for a polyatomic gas.
We assume $g^{eq}$ to be of the form:
\begin{align*}
g^{eq}_i &= \left( \AS{a}_1 \rho' + \AS{a}_2 \theta' + \AS{b} \vec{c}_i \cdot \vec{u}' + \frac{1}{2} |\vec{c}_i|^2 \left( \AS{c}_1 \rho' + \AS{c}_2 \theta' \right) \right) f^*_i
\end{align*}
Here it is important that due to the definition of the energy density $\mathcal{E}$ the temperature is now given by:
\begin{align*}
\theta' &= \frac{1}{\rho_0} \left( (\gamma-1) \langle \frac{1}{2} ( |\vec{v}|^2 + \beta ) g \rangle_\mathcal{S} - \theta_0 \rho' \right)
\end{align*}
By plugging $g^{eq}$ into the scaled FDVM equation (\ref{eq: FDVM equation}) and taking the moments with respect to $1$, $\vec{v}$, and $\frac{1}{2} \left(|\vec{v}|^2 + \beta \right)$ we can deduce 8 independent constraints on the choice of $\AS{a}_1$, $\AS{a}_2$, $\AS{b}$, $\AS{c}_1$, $\AS{c}_2$, $f^*$, and $\beta$ plus non-negativity constraints on $f^*$ and $\beta$.
An additional two constraint equations are given by taking the moments with respect to $1$ and $\frac{1}{2} \left( |\vec{v}|^2 + \beta \right)$ of $f^*$ which have to resemble the background density and the background energy density respectively.
One can now solve this system of constraint equations for deriving an equilibrium distribution along with a function $\beta$ for which the according FDVM model solves the polyatomic LEE.
In section \ref{sec: examples of velocity sets}, two schemes for diatomic gases in 2D and 3D are derived and analyzed.

\section{Derivation of the Lattice Boltzmann Equation}
\label{sec: LBE}

Using standard methodology (cf. \cite{Junk2005}) to fully discretize the semi-discrete FDVM equation (\ref{eq: FDVM equation}), we obtain the Lattice Boltzmann Equation (LBE):
\begin{equation}
\label{eq: LBE}
g_i(t+\epsilon,\vec{x}+\epsilon \vec{c}_i ) - g_i(t,\vec{x}) = \frac{1}{\tau} \left( g_i^{eq}(t,\vec{x}) - g_i(t,\vec{x}) \right).
\end{equation}
This means we get the fully discrete and explicit update scheme:
\begin{equation}
\label{eq: LBE update scheme}
g_i(t+\epsilon,\vec{x}+\epsilon \vec{c}_i ) = \left( 1-\frac{1}{\tau} \right) g_i(t,\vec{x})  + \frac{1}{\tau} g_i^{eq}(t,\vec{x}).
\end{equation}
This is the Lattice Boltzmann scheme.
Due to the linear structure of the equilibrium distributions $g_i^{eq}$ equation (\ref{eq: LBE update scheme}) can be written in linear form:
\begin{equation}
\label{eq: LBE linear equation}
\underline{\vec{g}}(t+\epsilon, \vec{x}+\epsilon \vec{c} ) = H(\tau) \underline{\vec{g}}(t,\vec{x}),
\end{equation}
with the following definitions:
$$
\underline{\vec{g}}(t,\vec{x}) = \left( g_1(t,\vec{x}), \dots, g_n(t,\vec{x}) \right)^T,\\
$$
$$
\underline{\vec{g}}(t+\epsilon,\vec{x}+\epsilon \vec{c}) = \left( g_1(t+\epsilon,\vec{x}+\epsilon \vec{c}_1), \dots, g_n(t+\epsilon,\vec{x}+\epsilon \vec{c}_n) \right)^T,\\
$$
\begin{align*}
H_{ij}(\tau) &= \left( 1 - \frac{1}{\tau} \right) \delta_{ij} + \frac{1}{\tau} \bigg( \AS{a}_1 + \frac{1}{2} |\vec{c}_i|^2 \AS{c}_1\\
& + \AS{b} \vec{c}_i \cdot \vec{c}_j + \frac{1}{\rho_0} \left( \frac{\gamma - 1}{2} \left( |\vec{c}_j|^2 + \beta_j \right) - \theta_0 \right) \left( \AS{a}_2 + \frac{1}{2} |\vec{c}_i|^2 \AS{c}_2 \right) \bigg) f^*_i.
\end{align*}

\section{Analysis of the LBE}
\label{sec: Analysis of the LBE}

In this section, we show second-order consistency of the LBE (\ref{eq: LBE}) and state conditions for $L^2$-stability of the method.
For analysis of consistency we follow the method of Junk \etal \cite{Junk2005}.
We assume a periodic domain $\Omega$ and for all discrete velocities $\vec{c}_i \in \mathcal{S}$ the following initial condition:
$$
g_i(t=0,\vec{x}) = g^{eq}( \rho'(t=0,\vec{x}), \vec{u}'(t=0,\vec{x}), \theta'(t=0,\vec{x}) ).
$$
In addition, we assume that the initial macroscopic conditions $\rho'(t=0,\vec{x})$, $\vec{u}'(t=0,\vec{x})$, and $\theta'(t=0,\vec{x})$ are independent of $\epsilon$.
We use the following regular expansions for the density $g_i$, the equilibrium $g_i^{eq}$, and the macroscopic variables $\rho'$, $\vec{u}'$, and $\theta'$:
\begin{align*}
g_i &= \sum_{i=0}^\infty \epsilon^i g_i^{(i)}, &&& g_i^{eq} &= \sum_{i=0}^\infty \epsilon^i g_i^{eq(i)},\\
\rho' &= \sum_{i=0}^\infty \epsilon^i \rho'^{(i)}, &&& \vec{u}' &= \sum_{i=0}^\infty \epsilon^i \vec{u}'^{(i)}, &&& \theta' &= \sum_{i=0}^\infty \epsilon^i \theta'^{(i)}.
\end{align*}
Hence, the initial conditions for the expansion coefficients $g_i^{(k)}$ are given as follows:
\begin{align*}
g_i^{(0)}(t=0,\vec{x}) &= g^{eq}( \rho'(t=0,\vec{x}), \vec{u}'(t=0,\vec{x}), \theta'(t=0,\vec{x}) ),&\\
g_i^{(k)}(t=0,\vec{x}) &= 0 &\hbox{ for } k \geq 1.
\end{align*}
We now plug these expansions into the LBE (\ref{eq: LBE}) and Taylor-expand the left hand side.
Sorting the terms according to powers of $\epsilon$ results in the following set of equations:
\begin{subequations}
\begin{align}
\epsilon^0: && 0 &= \frac{1}{\tau} \left( g_i^{eq(0)}(t,\vec{x}) - g_i^{(0)}(t,\vec{x}) \right) \label{eq: taylor eps to 0}\\
\epsilon^1: && \vec{c}_i \cdot \nabla_{\vec{x}} g^{(0)}_i(t,{\vec{x}}) + \partial_t g^{(0)}_i(t,{\vec{x}}) &= \frac{1}{\tau} \left( g_i^{eq(1)}(t,{\vec{x}}) - g^{(1)}_i(t,{\vec{x}}) \right), \label{eq: taylor eps to 1}\\
\epsilon^2: && \partial_t g^{(1)}_i(t,{\vec{x}}) + \vec{c}_i \cdot \nabla_{\vec{x}} g^{(1)}_i(t,{\vec{x}}) + \frac{1}{2} \vec{c}_i \otimes \vec{c}_i : \nabla_{\vec{x}} \nabla_{\vec{x}} g^{(0)}_i(t,{\vec{x}})\span\span\span\span\span\span\nonumber\\
&& + \partial_t \vec{c}_i \cdot \nabla_{\vec{x}} g^{(0)}_i(t,{\vec{x}}) +\frac{1}{2} \partial_t^2 g_i^{(0)}(t,{\vec{x}}) &= \frac{1}{\tau} \left( g_i^{eq(2)}(t,{\vec{x}}) - g^{(2)}_i(t,{\vec{x}}) \right). \label{eq: taylor eps to 2}
\end{align}
\end{subequations}
Here, we use tensor notation where $\otimes$ denotes the standard tensor product, $\cdot$ denotes a first order, and $:$ a second-order tensor contraction.
From equation (\ref{eq: taylor eps to 0}) we find: 
$$
g_i^{eq(0)}(t,{\vec{x}}) = g_i^{(0)}(t,{\vec{x}}).
$$
By computing the moments corresponding to $1$, $\vec{v}$, and $\frac{1}{2} \left( |\vec{v}|^2 + \beta \right)$ of equation (\ref{eq: taylor eps to 1}) we find:
\begin{align*}
\partial_t \rho'^{(0)} + \rho_0 \nabla_{\vec{x}} \cdot \vec{u}'^{(0)}  &= 0,\\
\rho_0 \partial_t \vec{u}'^{(0)}  + \rho_0 \nabla_{\vec{x}} \theta'^{(0)} + \theta_0 \nabla_{\vec{x}} \rho'^{(0)} &= \vec{0},\\
\frac{1}{\gamma-1} \rho_0 \partial_t \theta'^{(0)}  + \rho_0 \theta_0 \nabla_{\vec{x}} \cdot \vec{u}'^{(0)}  &= 0.
\end{align*}
Next, we analyze the moments corresponding to\\
$\zeta \in \left\{\vec{v}\to 1, \vec{v}\to v_1,..., \vec{v}\to v_D, \vec{v}\to \frac{1}{2} \left( |\vec{v}|^2 + \beta \right) \right\}$ of equation (\ref{eq: taylor eps to 2}).
The steps of this analysis can be found in \ref{app: analysis of the LBE}.
This analysis gives the following set of moments of equation (\ref{eq: taylor eps to 2}):
\begin{subequations}
\begin{equation}
\label{eq: eps to 2 zeroth}
0 = \partial_t \rho'^{(1)} + \rho_0 \nabla_{\vec{x}} \cdot \vec{u}'^{(1)},
\end{equation}
\medskip
\begin{equation}
\label{eq: eps to 2 first}
\begin{aligned}
\vec{0}=& \rho_0 \partial_t \vec{u}'^{(1)} + \rho_0 \nabla_{\vec{x}} \theta'^{(1)} + \theta_0 \nabla_{\vec{x}} \rho'^{(1)}\\
+& \left( \frac{1}{2} - \tau \right)  \nabla_{\vec{x}} \nabla_{\vec{x}} : \langle  \vec{v} \otimes \vec{v} \otimes \vec{v} g^{(0)} \rangle_\mathcal{S}\\
+&\left( \frac{1}{2} - \tau \right) \partial_t \nabla_{\vec{x}} \left( \rho_0 \theta'^{(0)} + \theta_0 \rho'^{(0)} \right),
\end{aligned}
\end{equation}
\medskip
\begin{equation}
\label{eq: eps to 2 second}
\begin{aligned}
0 =& \frac{1}{\gamma-1} \rho_0 \partial_t \theta'^{(1)} + \frac{1}{\gamma-1} \theta_0 \partial_t \rho'^{(1)} + \frac{\gamma}{\gamma-1} \rho_0 \theta_0 \nabla_{\vec{x}} \cdot \vec{u}'^{(1)}\\
+& \left( \frac{1}{2} - \tau \right) \frac{\gamma}{\gamma-1} \rho_0 \theta_0 \partial_t \nabla_{\vec{x}} \cdot \vec{u}'^{(0)}\\
+& \left( \frac{1}{2} - \tau \right) \nabla_{\vec{x}} \nabla_{\vec{x}} : \langle \frac{1}{2} \left( |\vec{v}|^2 + \beta \right) \vec{v} \otimes \vec{v} g^{(0)} \rangle_\mathcal{S}.
\end{aligned}
\end{equation}
\end{subequations}
With the choice $\tau = \frac{1}{2}$, equations (\ref{eq: eps to 2 zeroth}) - (\ref{eq: eps to 2 second}) resemble the LEE for the macroscopic variables $\rho'^{(1)}$, $\vec{u}'^{(1)}$, and $\theta'^{(1)}$.
The LEE together with the initial conditions $\rho'^{(1)}(t=0,\vec{x}) = 0$, $\vec{u}'^{(1)}(t=0,\vec{x}) = \vec{0}$, and $\theta'^{(1)}(t=0,\vec{x}) = 0$ result in the global solution for $t \geq 0$:
$$
\rho'^{(1)}(t,\vec{x}) = 0,\ \ \vec{u}'^{(1)}(t,\vec{x}) = \vec{0}, \ \ \theta'^{(1)}(t,\vec{x}) = 0.
$$
Therefore, for $\tau = \frac{1}{2}$ the LBM scheme is of at least second-order consistency.
\medskip
\\
In order to prove convergence one also needs to prove stability.
Here, we discuss $L^2$-stability utilizing a von Neumann analysis of the linear update equation (\ref{eq: LBE linear equation}) as described by LeVeque \cite{LeVeque2007} and Trefethen \cite{Trefethen1996}.
We would like to emphasize that Trefethen's work is the most thorough work on $L^2$-stability to our notice, incorporating important aspects of stability usually neglected in the literature.
In addition, we note that our collision operators for the velocity sets presented in \ref{sec: examples of velocity sets} also obtain stability structures as described by Banda \etal \cite{banda2006}, Junk and Yong \cite{JunkYong2009}, and Yong \cite{Yong2009862}.
\medskip
\\
Due to the results above, we fix $\tau = \frac{1}{2}$.
First, equation (\ref{eq: LBE linear equation}) is Fourier transformed in space to remove spatial dependency, resulting in the following equation:
$$
\underline{\hat{\vec{g}}}( \vec{k}, t + \epsilon ) = \Gamma( \vec{k}, \epsilon ) \underline{\hat{\vec{g}}}( \vec{k}, t ),
$$
with the wave vector $\vec{k}$, the Fourier transformed vector $\underline{\hat{\vec{g}}}$, and the matrix:
$$
\Gamma( \vec{k}, \epsilon ) = D( \vec{k}, \epsilon ) H( \tau = \frac{1}{2} ),
$$
where the matrix $D$ resulting from the spatial shift of the left hand side in equation (\ref{eq: LBE linear equation}) is of the form:
$$
D_{mn}( \vec{k}, \epsilon ) = \exp( -i \vec{k} \cdot \vec{c}_m \epsilon) \delta_{mn}.
$$
The following Theorem is used for proving stability of the LBM schemes:
\begin{theorem}[$L^2$-stability of LBM schemes]
\label{th: l2stability of LBM schemes}
If the matrix $\Gamma(\vec{k},\epsilon)$ fulfills the following conditions:
\begin{enumerate}
\item $\Gamma(\vec{k},\epsilon)$ is regular for all $\vec{k} \epsilon \in [-\pi,\pi]^D$,\label{cond: lemma regular}
\item the spectral radius $\rho(\Gamma(\vec{k},\epsilon)) \leq 1$ for all $\vec{k} \epsilon \in [-\pi,\pi]^D$, and \label{cond: lemma spectral radius}
\item there exists a finite constant $C$, s.t. $\kappa( V(\vec{k},\epsilon) ) \leq C$ for all $\vec{k} \epsilon \in [-\pi,\pi]^D$, where $\kappa$ denotes the condition number corresponding to the $L^2$-norm and $V(\vec{k},\epsilon)$ denotes the eigenvector matrix of $\Gamma(\vec{k},\epsilon)$.\label{cond: lemma condition number}
\end{enumerate} 
Then the according LBM scheme is $L^2$-stable, i.e. the $L^2$-norm of the solution does not explode in finite time.
\end{theorem}

The proof of Theorem \ref{th: l2stability of LBM schemes} uses the stability Theorem 4.11 stated by Trefethen in \cite{Trefethen1996}.
We restate this theorem here but adapt it to our notation.
\begin{theorem}[Stability via the Kreiss Matrix Theorem]
\label{th: Kreiss}
A linear, constant-coefficient finite difference formula is stable in the $L^2$-norm if and only if	
\begin{equation}
\rho_\alpha(\Gamma(\vec{k},\epsilon)) \leq 1 + \mathcal{O}(\alpha) + \mathcal{O}(\epsilon)
\end{equation}
for $\vec{k} \in \left[-\frac{\pi}{\epsilon},\frac{\pi}{\epsilon}\right]$ as $\alpha \to 0$ and $\epsilon \to 0$.
\end{theorem}
Here $\rho_\alpha( B )$ denotes the $\alpha$-pseudospectral radius of a matrix $B$ which is defined as follows:
$$
\rho_\alpha( B ) = \sup_{||A||_2 \leq \alpha} \rho( B + A )
$$
and $\rho(C)$ denotes the spectral radius of the matrix $C$.
\\
Using Theorem \ref{th: Kreiss} we can now prove Theorem \ref{th: l2stability of LBM schemes}.

\begin{proof}[Proof of Theorem \ref{th: l2stability of LBM schemes}]
As stated in Theorem \ref{th: Kreiss} the analysis has to hold for all $\vec{k} \in \left[-\frac{\pi}{\epsilon},\frac{\pi}{\epsilon}\right]^D$.
However, since $\vec{k}$ occurs in $\Gamma(\vec{k},\epsilon)$ only within products $k_i \epsilon \ (i=1..D)$ it is sufficient to do the analysis for $\vec{k} \epsilon \in [-\pi, \pi]^D$.
As result of condition \ref{cond: lemma regular} we can diagonalize $\Gamma(\vec{k},\epsilon)$ as $\Gamma(\vec{k},\epsilon) = (V(\vec{k},\epsilon))^{-1} \Lambda( \vec{k}, \epsilon) V(\vec{k},\epsilon)$.
Here $\Lambda(\vec{k},\epsilon)$ is a diagonal matrix with $\Lambda(\vec{k},\epsilon)_{ii}$ being the $i$-th eigenvalue $\lambda_i$ of $\Gamma(\vec{k},\epsilon)$ and $V(\vec{k},\epsilon)$ being the eigenvector matrix of $\Gamma(\vec{k},\epsilon)$, i.e. column $i$ of $V(\vec{k},\epsilon)$ is the eigenvector of $\Gamma(\vec{k},\epsilon)$ corresponding to the eigenvalue $\lambda_i$.
Due to the diagonalizability of $\Gamma(\vec{k},\epsilon)$ we can use the theorem of Bauer-Fike \cite{bauer1960} to estimate the $\alpha$-pseudospectral radius of $\Gamma(\vec{k},\epsilon)$ as follows:
\begin{equation}
\label{eq: estimate Bauer-Fike}
\begin{aligned}
\rho_\alpha(\Gamma(\vec{k},\epsilon)) &= \sup_{||A||_2 \leq\alpha} \rho( \Gamma(\vec{k},\epsilon) + A )\\
&\leq \rho(\Gamma(\vec{k},\epsilon) ) + \sup_{||A||_2 \leq\alpha} \kappa( V( \vec{k}, \epsilon ) ) ||A||_2\\
&\leq \rho(\Gamma(\vec{k},\epsilon) ) + \kappa( V( \vec{k}, \epsilon ) ) \alpha
\end{aligned}
\end{equation}
Finally, with assumptions \ref{cond: lemma spectral radius}. and \ref{cond: lemma condition number}. for all $\vec{k} \epsilon \in [-\pi,\pi]^D$ we find:
\begin{align*}
\rho_\alpha(\Gamma(\vec{k},\epsilon)) &\leq 1 + \mathcal{O}(\alpha)\\
&\leq 1 + \mathcal{O}(\alpha) + \mathcal{O}(\epsilon)
\end{align*}
Therefore, Theorem \ref{th: Kreiss} proves $L^2$-stability of such LBM schemes.
\end{proof}

For schemes for which the matrix $\Gamma(\vec{k},\epsilon)$ is normal or unitary the following corollaries can be derived from Theorem \ref{th: l2stability of LBM schemes}:
\begin{corollary}[$L^2$-stability of LBM schemes where $\Gamma(\vec{k},\epsilon)$ is normal]
If the matrix $\Gamma(\vec{k},\epsilon)$ is normal for all $\vec{k}\epsilon \in [-\pi,\pi]^D$ and fulfills conditions \ref{cond: lemma regular}. and \ref{cond: lemma spectral radius}. of Theorem \ref{th: l2stability of LBM schemes} then the LBM scheme is $L^2$-stable.
\end{corollary}

\begin{proof}
Normal matrices are unitarily diagonalizable.
Due to this, we know that the eigenvector matrix $V(\vec{k},\epsilon)$ of $\Gamma(\vec{k},\epsilon)$ is unitary for all $\vec{k} \epsilon \in [-\pi,\pi]$.
Therefore, the condition number $\kappa( \Gamma(\vec{k},\epsilon)) = 1$ for all $\vec{k} \epsilon \in [-\pi,\pi]$.
Together with conditions \ref{cond: lemma regular}. and \ref{cond: lemma spectral radius}. of Theorem \ref{th: l2stability of LBM schemes} this shows $L^2$-stability of the LBM scheme.
\end{proof}

\begin{corollary}[$L^2$-stability of LBM schemes where $\Gamma(\vec{k},\epsilon)$ is unitary]
If the matrix $\Gamma(\vec{k},\epsilon)$ is unitary for all $\vec{k} \epsilon \in [-\pi,\pi]^D$  then the LBM scheme is $L^2$-stable.
\end{corollary}

\begin{proof}
We know that unitary matrices are normal, regular, and have spectral radius $1$.
Therefore,  the LBM scheme is $L^2$-stable.
\end{proof}

\section{Examples of Velocity Sets}
\label{sec: examples of velocity sets}

In this section, different finite discrete velocity sets for which the LBM is converging are analyzed.
Here, it is important to remember that due to the dedimensionalization it is not necessary to provide discrete velocity sets for every choice of dimensional background stream parameters $\ub{\rho_0}$ and $\ub{\theta_0}$.

\subsection{D1Q3}
\label{sec: D1Q3}

For the one dimensional case we choose $\rho_0=1$ and $\theta_0=\frac{1}{3}$.
For this choice we find the compact velocity set $\mathcal{S} = \left\{ \vec{c}_i: i=1...3 \right\}$ we call D1Q3 with:
$$
\vec{c}_1 = 0,\ \vec{c}_2 = -1,\ \vec{c}_3 = 1
$$
and
$$
f^*(\vec{c}_i) = \left\{
\begin{array}{cl}
\frac{2}{3} & \hbox{for } i=1,\\
\frac{1}{6} & \hbox{for } i = 2, 3.
\end{array}
\right.
$$
The matrix $\Gamma(\vec{k},\epsilon)$ for D1Q3 is given as:
$$
\Gamma(\vec{k},\epsilon) = \begin{pmatrix}
1 & 0 & 0\\
0 & e^{i \vec{k} \epsilon} & 0\\
0 & 0 & e^{-i \vec{k} \epsilon}
\end{pmatrix}.
$$
One can easily see that $\Gamma(\vec{k},\epsilon)$ is unitary for all $\vec{k} \epsilon \in [ -\pi,\pi]$.
Therefore, the LBM scheme with the D1Q3 velocity set is $L^2$-stable.

\subsection{2D Velocity Sets}

\subsubsection{D2Q5 for Monoatomic Gases}

For the two dimensional case we choose $\rho_0 = 1$ and $\theta_0 = \frac{1}{4}$.
For this choice we obtain the compact velocity set $\mathcal{S} = \left\{ \vec{c}_i: i=1...5 \right\}$ with:
$$
\vec{c}_1 = \vector{0\\0}, \vec{c}_2 = \vector{-1\\0}, \vec{c}_3 = \vector{1\\0}, \vec{c}_4 = \vector{0\\-1}, \vec{c}_5 = \vector{0\\1},
$$
and
$$
f^*(\vec{c}_i) = \left\{
\begin{array}{cl}
\frac{1}{2} & \hbox{for } i=1,\\
\frac{1}{8} & \hbox{for } i=2...5.
\end{array}
\right.
$$
One can easily check that $\Gamma( \vec{k}, \epsilon)$ is unitary for all choices of $\vec{k}$ and $\epsilon$.
Therefore, the LBM scheme using D2Q5 is stable for monoatomic gases.

\subsubsection{D2Q5 for Diatomic Gases}

For the setting of diatomic gases with background density $\rho_0 = \frac{20}{3}$ and background temperature $\theta_0 = \frac{3}{10}$, we find the D2Q5 velocity set with:
\begin{align*}
f_i &= \left\{
\begin{array}{cl}
\frac{8}{3} &\hbox{ for } i = 0\\
1 & \hbox{ for } i = 1..4
\end{array}
\right., &
\beta_i &= \left\{
\begin{array}{cl}
0 &\hbox{ for } i = 0\\
\frac{1}{2} & \hbox{ for } i = 1..4
\end{array}
\right.
\end{align*}
\begin{align*}
a_1 &= \frac{1}{\rho_0}, & a_2 &= -5, & b &= \frac{1}{\theta_0}, & c_1 &= 0, & c_2 &= \frac{5}{\theta_0}
\end{align*} 
With this velocity set, the matrix $\Gamma(\tau = \frac{1}{2})$ is unitary.
Therefore, the LBM scheme using D2Q5 is stable for diatomic gases.

\subsection{3D}

For the 3D case we analyze a family of velocity sets.
We start with the large velocity set $\bar{\mathcal{S}} = \{ \vec{c}_i: i=1..27 \}$ where we have:
$$
c_i = \left\{
\begin{array}{ll}
(0,0,0) & i=1\\
(\pm 1,0,0),(0,\pm 1,0),(0,0,\pm 1) & i=2..7\\
(\pm 1, \pm 1, 0), (\pm 1, 0, \pm 1), (0,\pm 1, \pm 1) & i=8..19\\
(\pm 1, \pm 1, \pm 1) & i=20..27
\end{array}\right. .
$$

\subsubsection{Monoatomic Gases}

For monoatomic gases $f^*$ is of the following form:
$$
f^* = \left\{
\begin{array}{ll}
\frac{1}{2} \rho_0 \theta_0 ( 15\theta_0 - 9 ) + \rho_0 - 8 \alpha & \hbox{ for } \vec{c}_1\\
\frac{1}{2} \rho_0 \theta_0 ( 2 - 5 \theta_0 ) + 4 \alpha & \hbox{ for } \vec{c}_i, i = 2..7\\
\frac{1}{8} \rho_0 \theta_0 ( 5 \theta_0 - 1 ) - 2 \alpha & \hbox{ for } \vec{c}_i, i = 8..19\\
\alpha & \hbox{ for } \vec{c}_i, i = 20..27
\end{array}
\right.
$$
In the following the 3D stencils shown in figure \ref{fig: 3D velocity sets} are presented.
\medskip
\\
For the choice $\rho_0 = 1$, $\theta_0 = \frac{1}{5}$, and $\alpha = 0$ the function $f^*$ is zero for velocities $\vec{c}_i: i=8..27$.
Therefore, the velocity set $\bar{\mathcal{S}}$ reduces to the extremely compact set $\mathcal{S} = \{ \vec{c}_i: i=1..7 \}$ which we call D3Q7.
For the D3Q7 velocity set we see that the matrix $\Gamma(\vec{k},\epsilon)$ is unitary for all $\vec{k} \in \mathbb{R}^3$ and $\epsilon \in \mathbb{R}$.
Therefore the LBM with the D3Q7 velocity set is stable.
\medskip
\\
For the choice $\rho_0 = 1$, $\theta_0 = \frac{3}{5}$, and $\alpha = \frac{3}{40}$ the function $f^*$ is zero for velocities $\vec{c}_i: i=2..19$.
Therefore, the velocity set $\bar{\mathcal{S}}$ reduces to the slightly larger set $\mathcal{S} = \{ \vec{c}_i: i\in\{1,20..27\} \}$ which we call D3Q9.
As for the D3Q7 velocity set the matrix $\Gamma(\vec{k},\epsilon)$ is unitary for all $\vec{k} \in \mathbb{R}^3$ and $\epsilon \in \mathbb{R}$.
Therefore the LBM with the D3Q9 velocity set is stable.
\medskip
\\
For the choice $\rho_0 = 1$, $\theta_0 = \frac{2}{5}$, and $\alpha = 0$ the function $f^*$ is zero for velocities $\vec{c}_i: i\in\{2..7, 20..27\}$.
Therefore, the velocity set $\bar{\mathcal{S}}$ reduces to the set $\mathcal{S} = \{ \vec{c}_i: i\in\{1, 8..19\} \}$ which we call D3Q13.
As for the D3Q7 velocity set the matrix $\Gamma(\vec{k},\epsilon)$ is unitary for all $\vec{k} \in \mathbb{R}^3$ and $\epsilon \in \mathbb{R}$.
Therefore the LBM with the D3Q13 velocity set is stable.
\medskip
\\
For the choice $\rho_0 = 1$, $\theta_0 = \frac{3}{10}$, and $\alpha = 0$ the function $f^*$ is zero for velocities $\vec{c}_i: i = 20..27$.
Therefore, the velocity set $\bar{\mathcal{S}}$ reduces to the set $\mathcal{S} = \{ \vec{c}_i: i = 1..19 \}$ which we call D3Q19.
For this velocity set our analytical stability criterion could be proven.
But, following Banda \etal \cite{banda2006}, one can find a stability structure for the collision operator generated by this velocity set.
One can easily verify that the matrix $H(\tau=1)$, which is simply the equilibrium function in matrix form, is a projection matrix.
Therefore, the eigenvalues of $H(\tau=\frac{1}{2})$ are $0$ and $-2$.
In addition, we find that with the positive definite diagonal matrix $A_0 = \diag(3,13 I_6, 52 I_12)$ the matrix $A_0 H(\tau=\frac{1}{2})$ is symmetric.
Here, $I_n$ denotes the unity matrix in $n$ dimension.
Thus, there exists an invertible matrix $P$ such that
\begin{align*}
A_0 &= P^T P \quad \text{ and } \quad A_0 H(\tau=\frac{1}{2}) = P^T \Lambda P,
\end{align*}
where we can assume that
\begin{align*}
\Lambda = -\diag(0,0,0,0,0,2 I_{14})
\end{align*}
Therefore, the D3Q19 velocity set admits a stability structure and is stable. 
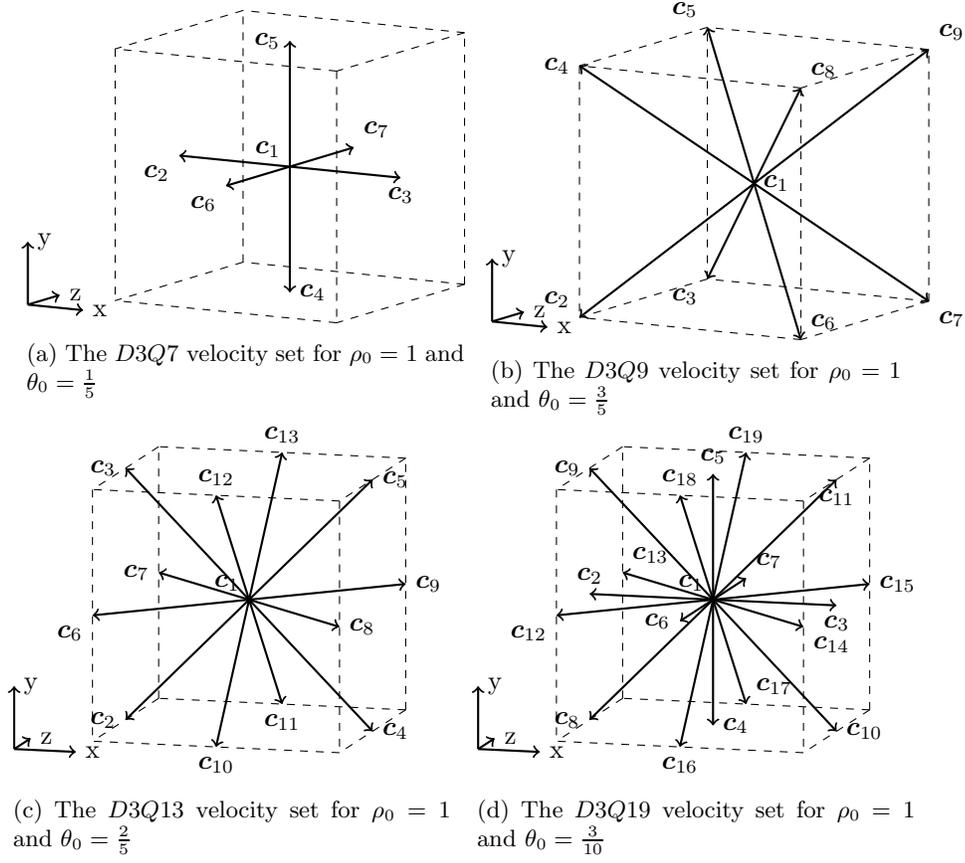
\begin{figure}[ht]
\begin{center}
\begin{subfigure}[t!]{0.48\textwidth}
\tdplotsetmaincoords{80}{120}
\begin{tikzpicture}[tdplot_main_coords, scale=.85]
\draw[thick,->] (0,0,0) -- (-2,0,0) node[anchor=south west]{$\vec{c}_7$};
\draw[thick,->] (0,0,0) -- (2,0,0) node[anchor=north east]{$\vec{c}_6$};
\draw[thick,->] (0,0,0) -- (0,-2,0) node[anchor=north east]{$\vec{c}_2$};
\draw[thick,->] (0,0,0) -- (0,2,0) node[anchor=north]{$\vec{c}_3$};
\draw[thick,->] (0,0,0) -- (0,0,-2) node[anchor=west]{$\vec{c}_4$};
\draw[thick,->] (0,0,0) -- (0,0,2) node[anchor=east]{$\vec{c}_5$};
\draw[] (0,0,0) -- (0,0,0) node[anchor=south east]{$\vec{c}_1$};
\draw[dashed] (-2,-2,-2) -- (2, -2, -2 );
\draw[dashed] (-2,-2,-2) -- (-2, -2, 2 );
\draw[dashed] (-2,-2,2) -- (2, -2, 2 );
\draw[dashed] (2,-2,2) -- (2, -2, -2 );
\draw[dashed] (-2,2,-2) -- (2, 2, -2 );
\draw[dashed] (-2,2,-2) -- (-2,2, 2 );
\draw[dashed] (-2,2,2) -- (2,2, 2 );
\draw[dashed] (2,2,2) -- (2, 2, -2 );
\draw[dashed] (2,2,2) -- (2, -2, 2 );
\draw[dashed] (2,2,-2) -- (2, -2, -2 );
\draw[dashed] (-2,2,2) -- (-2, -2, 2 );
\draw[dashed] (-2,2,-2) -- (-2, -2, -2 );
\draw[thick,->] (3,-3,-2) -- (2,-3,-2) node [anchor=west]{z};
\draw[thick,->] (3,-3,-2) -- (3,-3,-1) node [anchor=west]{y};
\draw[thick,->] (3,-3,-2) -- (3,-2,-2) node [anchor=west]{x};
\end{tikzpicture}
\caption{The $D3Q7$ velocity set for $\rho_0 = 1$ and $\theta_0 = \frac{1}{5}$}
\label{fig: D3Q7}
\end{subfigure}
~
\begin{subfigure}[t!]{0.45\textwidth}
\tdplotsetmaincoords{80}{120}
\begin{tikzpicture}[tdplot_main_coords, scale=.85]
\draw[thick,->] (0,0,0) -- (-2,-2,-2) node[anchor=north east]{$\vec{c}_3$};
\draw[thick,->] (0,0,0) -- (-2,-2,2) node[anchor=south east]{$\vec{c}_5$};
\draw[thick,->] (0,0,0) -- (-2,2,-2) node[anchor=north west]{$\vec{c}_7$};
\draw[thick,->] (0,0,0) -- (-2,2,2) node[anchor=south west]{$\vec{c}_9$};
\draw[thick,->] (0,0,0) -- (2,-2,-2) node[anchor=south east]{$\vec{c}_2$};
\draw[thick,->] (0,0,0) -- (2,-2,2) node[anchor=east]{$\vec{c}_4$};
\draw[thick,->] (0,0,0) -- (2,2,-2) node[anchor=south west]{$\vec{c}_6$};
\draw[thick,->] (0,0,0) -- (2,2,2) node[anchor=south west]{$\vec{c}_8$};
\draw[] (0,0,0) -- (0,0,0) node[anchor=west]{$\vec{c}_1$};
\draw[dashed] (-2,-2,-2) -- (2, -2, -2 );
\draw[dashed] (-2,-2,-2) -- (-2, -2, 2 );
\draw[dashed] (-2,-2,2) -- (2, -2, 2 );
\draw[dashed] (2,-2,2) -- (2, -2, -2 );
\draw[dashed] (-2,2,-2) -- (2, 2, -2 );
\draw[dashed] (-2,2,-2) -- (-2,2, 2 );
\draw[dashed] (-2,2,2) -- (2,2, 2 );
\draw[dashed] (2,2,2) -- (2, 2, -2 );
\draw[dashed] (2,2,2) -- (2, -2, 2 );
\draw[dashed] (2,2,-2) -- (2, -2, -2 );
\draw[dashed] (-2,2,2) -- (-2, -2, 2 );
\draw[dashed] (-2,2,-2) -- (-2, -2, -2 );
\draw[thick,->] (3,-3,-2) -- (2,-3,-2) node [anchor=west]{z};
\draw[thick,->] (3,-3,-2) -- (3,-3,-1) node [anchor=west]{y};
\draw[thick,->] (3,-3,-2) -- (3,-2,-2) node [anchor=west]{x};
\end{tikzpicture}
\caption{The $D3Q9$ velocity set for $\rho_0 = 1$ and $\theta_0 = \frac{3}{5}$}
\end{subfigure}
\\
\begin{subfigure}[t!]{0.48\textwidth}
\tdplotsetmaincoords{80}{105}
\begin{tikzpicture}[tdplot_main_coords, scale=.85]
\draw[thick,->] (0,0,0) -- (-2,-2,0) node[anchor=east]{$\vec{c}_{7}$};
\draw[thick,->] (0,0,0) -- (-2,2,0) node[anchor=west]{$\vec{c}_{9}$};
\draw[thick,->] (0,0,0) -- (2,-2,0) node[anchor=north east]{$\vec{c}_{6}$};
\draw[thick,->] (0,0,0) -- (2,2,0) node[anchor=west]{$\vec{c}_{8}$};
\draw[thick,->] (0,0,0) -- (-2,0,-2) node[anchor=north]{$\vec{c}_{11}$};
\draw[thick,->] (0,0,0) -- (-2,0,2) node[anchor=south]{$\vec{c}_{13}$};
\draw[thick,->] (0,0,0) -- (2,0,-2) node[anchor=north]{$\vec{c}_{10}$};
\draw[thick,->] (0,0,0) -- (2,0,2) node[anchor=south]{$\vec{c}_{12}$};
\draw[thick,->] (0,0,0) -- (0,-2,-2) node[anchor=east]{$\vec{c}_{2}$};
\draw[thick,->] (0,0,0) -- (0,-2,2) node[anchor=east]{$\vec{c}_{3}$};
\draw[thick,->] (0,0,0) -- (0,2,-2) node[anchor=west]{$\vec{c}_{4}$};
\draw[thick,->] (0,0,0) -- (0,2,2) node[anchor=west]{$\vec{c}_{5}$};
\draw[] (0,0,0) -- (0,0,0) node[anchor=south east]{$\vec{c}_1$};
\draw[dashed] (-2,-2,-2) -- (2, -2, -2 );
\draw[dashed] (-2,-2,-2) -- (-2, -2, 2 );
\draw[dashed] (-2,-2,2) -- (2, -2, 2 );
\draw[dashed] (2,-2,2) -- (2, -2, -2 );
\draw[dashed] (-2,2,-2) -- (2, 2, -2 );
\draw[dashed] (-2,2,-2) -- (-2,2, 2 );
\draw[dashed] (-2,2,2) -- (2,2, 2 );
\draw[dashed] (2,2,2) -- (2, 2, -2 );
\draw[dashed] (2,2,2) -- (2, -2, 2 );
\draw[dashed] (2,2,-2) -- (2, -2, -2 );
\draw[dashed] (-2,2,2) -- (-2, -2, 2 );
\draw[dashed] (-2,2,-2) -- (-2, -2, -2 );
\draw[thick,->] (3,-3,-2) -- (2,-3,-2) node [anchor=west]{z};
\draw[thick,->] (3,-3,-2) -- (3,-3,-1) node [anchor=west]{y};
\draw[thick,->] (3,-3,-2) -- (3,-2,-2) node [anchor=west]{x};
\end{tikzpicture}
\caption{The $D3Q13$ velocity set for $\rho_0 = 1$ and $\theta_0 = \frac{2}{5}$}
\label{fig: D3Q13}
\end{subfigure}
~
\begin{subfigure}[t!]{0.48\textwidth}
\tdplotsetmaincoords{80}{105}
\begin{tikzpicture}[tdplot_main_coords, scale=.85]
\draw[thick,->] (0,0,0) -- (-2,0,0) node[anchor=south west]{$\vec{c}_7$};
\draw[thick,->] (0,0,0) -- (2,0,0) node[anchor=east]{$\vec{c}_6$};
\draw[thick,->] (0,0,0) -- (0,-2,0) node[anchor=south]{$\vec{c}_2$};
\draw[thick,->] (0,0,0) -- (0,2,0) node[anchor=north]{$\vec{c}_3$};
\draw[thick,->] (0,0,0) -- (0,0,-2) node[anchor=west]{$\vec{c}_4$};
\draw[thick,->] (0,0,0) -- (0,0,2) node[anchor=south]{$\vec{c}_5$};
\draw[] (0,0,0) -- (0,0,0) node[anchor=south east]{$\vec{c}_1$};
\draw[thick,->] (0,0,0) -- (-2,-2,0) node[anchor=south west]{$\vec{c}_{13}$};
\draw[thick,->] (0,0,0) -- (-2,2,0) node[anchor=west]{$\vec{c}_{15}$};
\draw[thick,->] (0,0,0) -- (2,-2,0) node[anchor=north east]{$\vec{c}_{12}$};
\draw[thick,->] (0,0,0) -- (2,2,0) node[anchor=north west]{$\vec{c}_{14}$};
\draw[thick,->] (0,0,0) -- (-2,0,-2) node[anchor=south west]{$\vec{c}_{17}$};
\draw[thick,->] (0,0,0) -- (-2,0,2) node[anchor=south]{$\vec{c}_{19}$};
\draw[thick,->] (0,0,0) -- (2,0,-2) node[anchor=north]{$\vec{c}_{16}$};
\draw[thick,->] (0,0,0) -- (2,0,2) node[anchor=south]{$\vec{c}_{18}$};
\draw[thick,->] (0,0,0) -- (0,-2,-2) node[anchor=east]{$\vec{c}_{8}$};
\draw[thick,->] (0,0,0) -- (0,-2,2) node[anchor=east]{$\vec{c}_{9}$};
\draw[thick,->] (0,0,0) -- (0,2,-2) node[anchor=west]{$\vec{c}_{10}$};
\draw[thick,->] (0,0,0) -- (0,2,2) node[anchor=north]{$\vec{c}_{11}$};
\draw[dashed] (-2,-2,-2) -- (2, -2, -2 );
\draw[dashed] (-2,-2,-2) -- (-2, -2, 2 );
\draw[dashed] (-2,-2,2) -- (2, -2, 2 );
\draw[dashed] (2,-2,2) -- (2, -2, -2 );
\draw[dashed] (-2,2,-2) -- (2, 2, -2 );
\draw[dashed] (-2,2,-2) -- (-2,2, 2 );
\draw[dashed] (-2,2,2) -- (2,2, 2 );
\draw[dashed] (2,2,2) -- (2, 2, -2 );
\draw[dashed] (2,2,2) -- (2, -2, 2 );
\draw[dashed] (2,2,-2) -- (2, -2, -2 );
\draw[dashed] (-2,2,2) -- (-2, -2, 2 );
\draw[dashed] (-2,2,-2) -- (-2, -2, -2 );
\draw[thick,->] (3,-3,-2) -- (2,-3,-2) node [anchor=west]{z};
\draw[thick,->] (3,-3,-2) -- (3,-3,-1) node [anchor=west]{y};
\draw[thick,->] (3,-3,-2) -- (3,-2,-2) node [anchor=west]{x};
\end{tikzpicture}
\caption{The $D3Q19$ velocity set for $\rho_0 = 1$ and $\theta_0 = \frac{3}{10}$}
\label{fig: D3Q19}
\end{subfigure}
\end{center}
\caption{Different velocity sets in three dimensions.}
\label{fig: 3D velocity sets}
\end{figure}

\subsubsection{Diatomic Gases}

For diatomic gases we find the following family of D3Q7 velocity sets:
\begin{align*}
\rho_0 &= \frac{42}{5} f_1, & \theta_0 &= \frac{5}{21},
\end{align*}
\begin{align*}
a_1 &= \frac{1}{\rho_0}, & a_2 &= -\frac{21}{2}, & b &= \frac{21}{5}, & c_1 &= 0, & c_2 &= \frac{147}{5},
\end{align*}
\begin{align*}
f_i &= \left\{ \begin{array}{cl}
\frac{12}{5} f_1 & \hbox{ for } i = 0\\
f_1 &\hbox{ for } i = 1..6
\end{array}
\right.
&
\beta_i &= \left\{ \begin{array}{cl}
0 & \hbox{ for } i = 0\\
\frac{2}{3} &\hbox{ for } i = 1..6
\end{array}
\right.
\end{align*}
For the choice $f_1 = \frac{5}{42}$ we find $\Gamma(f_1=\frac{5}{42},\tau=\frac{1}{2})$ to be unitary.
Therefore, the LBM for diatomic gases with the D3Q7 velocity set is $L^2$-stable.

\section{Numerical Results}
\label{sec: Numerical Results}

In this section we present numerical results for different LBM schemes introduced in section \ref{sec: examples of velocity sets}.
Again, we would like to stress the freedom introduced by the dedimensionalization.
Due to this freedom, it is possible to simulate settings with different background flows using a single velocity set with its according background flow.
\medskip
\\
The test problem used in this section consists of the propagation of a Gauss pulse in the gas.
All results presented are for periodic domains and boundary and initial conditions.
For 1D we compare the results to the analytical solution.
For the mono- and diatomic cases in 2D and the diatomic case in 3D we compare the results to a reference solution obtained using the Finite Volume Method (FVM).
These FVM solutions were generated using the Clawpack \cite{clawpack} \cite{leveque1997} \cite{LangsethLeVeque00}, PyClaw \cite{pyclaw} \cite{pyclaw-sisc}, and PETSc \cite{petsc-web-page} \cite{petsc-user-ref} \cite{petsc-efficient} software packages with appropriate, self-written Riemann solvers.
Due to the enormous computational costs of obtaining the FVM solution in the 3D case and the little practical relevance of the monoatomic case, we perform a simple convergence analysis against a highly-resolved LBM solution for the 3D monoatomic case.
For the 1D case we measure the error using a discrete $L^2$-norm over the space-time domain $\mathbb{X} \times \mathbb{T}$.
For a scalar function $\eta$ defined at the grid points $\vec{p}_i$ and time steps $t_i$, this norm is defined as:
\begin{equation}
\label{eq: L2 space time norm}
|| \eta ||^2_{L^2(\mathbb{X} \times \mathbb{T})} = \sum_{i = 0}^{N_\mathbb{T}} \sum_{j = 1}^{N_\mathbb{X}} | \eta( t_i, \vec{p}_j ) |^2 \ \d t^{D+1}
\end{equation}
where $N_\mathbb{T}$, $N_\mathbb{X}$, and $\d t$ denote the number of time steps and grid points and the step size in both time and space respectively.
For the 2D and 3D cases we use the discrete $L^2$-norm at time $t=1$ over space only.
For a scalar function $\eta$ defined at the grid points $\vec{p}_i$ at time $t=T$, this norm is defined as:
\begin{equation}
\label{eq: L2 space norm}
|| \eta ||^2_{L^2(\mathbb{X}, t=T)} = \sum_{j = 1}^{N_\mathbb{X}} | \eta( t = 2, \vec{p}_j ) |^2 \ \d t^{D}
\end{equation}
where $N_\mathbb{X}$  and $\d t$ denote the number of grid points and the step size in both time and space respectively.
If not denoted otherwise, all results were computed non-dimensional.

\subsection{1D}

For the 1D case, we analyze the propagation of a Gauss pulse on the time-space domain $\mathbb{T} \times \mathbb{X} = [0;T] \times [0;1)$ for the D1Q3 velocity set with $\rho_0 = 0$ and $\theta_0 = \frac{1}{3}$.
The Gauss pulse is given by the following initial values:
\begin{subequations}
\label{ic: 1D gauss}
\begin{align}
\rho'(t=0, x ) &= \exp\left( - 100 \left( x - \frac{1}{2} \right)^2 \right),\\
u'_x(t=0, x ) &= 0,\\
\theta'(t=0, x ) &= 0.
\end{align}
\end{subequations}
Figure \ref{fig: L2-dens 1D gauss} shows the convergence behavior of the error in density for different end times $T$ measured in norm (\ref{eq: L2 space time norm}).
One can easily see that the errors are close to machine precision.
The reason for this is that the characteristic velocities of the 1D LEE with $\rho_0=1$ and $\theta_0=\frac{1}{3}$ coincide with the velocities within the D1Q3 velocity set.
Therefore, the D1Q3 LBM scheme can solve the 1D LEE with $\rho_0=1$ and $\theta_0=\frac{1}{3}$ analytically with respect to machine precision.

\begin{figure}[ht]
\begin{center}
\input{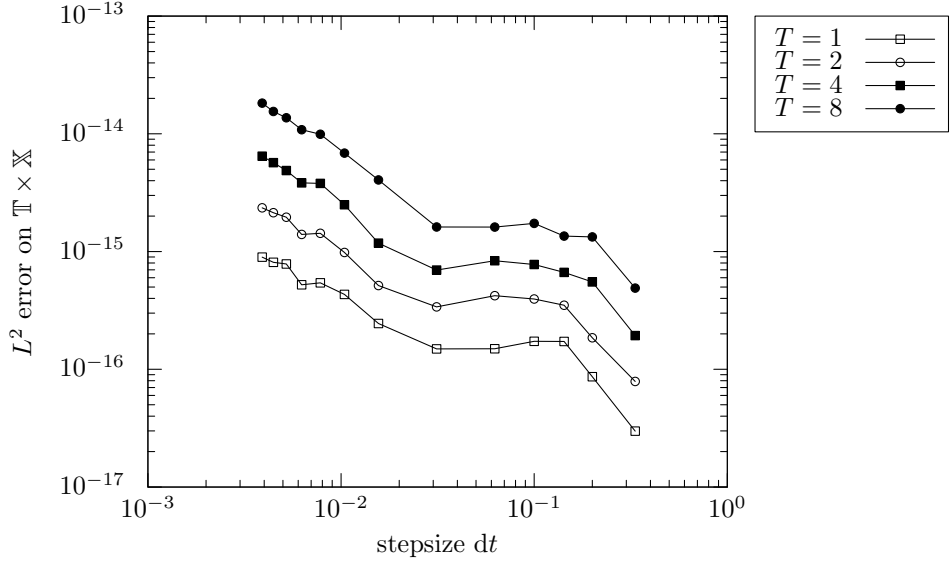}
\end{center}
\caption{$L^2$-error in the density on the time-space domain $\mathbb{T} \times \mathbb{X} = [0;T] \times [0;1)$ for the D1Q3 velocity set and initial conditions (\ref{ic: 1D gauss}).}
\label{fig: L2-dens 1D gauss}
\end{figure}

\subsection{2D}

We consider a 2D Gauss pulse in a monoatomic gas on the time-space domain $\mathbb{T} \times \mathbb{X}= [0;1] \times [0;2)^2$ for the D2Q5 velocity set with $\rho_0 = 1$ and $\theta_0 = \frac{1}{4}$ and initial conditions:
\begin{subequations}
\label{ic: 2D gauss}
\begin{align}
\rho'(t=0, x, y ) &= \exp\left( -7 \left|\left|(x,y)^T - (1,1)^T\right|\right|^2 \right), \\
u'_x(t=0, x, y ) &= 0,\\
u'_y( t=0, x, y ) &= 0,\\
\theta'(t=0, x, y ) &= 0.
\end{align}
\end{subequations}
Figure \ref{fig: L2-dens 2D gauss} presents the convergence behavior of the error in the macroscopic quantities, measured in the norm (\ref{eq: L2 space norm}).
Here, the FVM solution used as reference solution was computed using $5\dot 10^3 \times 5\cdot 10^3$ volumes.
One can easily verify a second-order convergence in this graph.
\begin{figure}[ht]
\begin{center}
\input{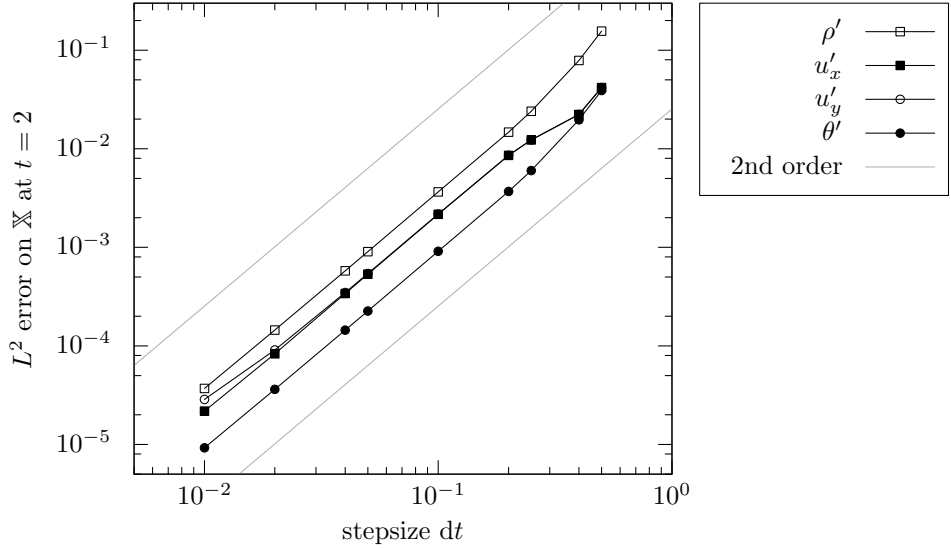}
\end{center}
\caption{$L^2$-error in the macroscopic variables on the spatial domain $\mathbb{X} = [0;2)^2$ at time $t = 2$ for the D2Q5 velocity set and initial conditions (\ref{ic: 2D gauss}).}
\label{fig: L2-dens 2D gauss}
\end{figure}
\medskip
\\
For a diatomic gas with background density $\rho_0 = \frac{20}{3}$ and background temperature $\theta_0 = \frac{3}{10}$ we simulate a Gauss pulse on the domain $\mathbb{T} \times \mathbb{X} = [0;1] \times [0;2)^2$ with initial conditions (\ref{ic: 2D gauss}).
Figure \ref{fig: L2-dens 2D gauss diatomic} shows the convergence behavior of the error in the macroscopic quantities measured in norm (\ref{eq: L2 space norm}).
Since the diatomic case is the more realistic case, the FVM solution used as reference solution was calculated using the higher resolution of $2\cdot 10^{4} \times 2\cdot 10^{4}$ volumes.
From figure \ref{fig: L2-dens 2D gauss diatomic} one can find second-order convergence.
\begin{figure}[ht]
\begin{center}
\input{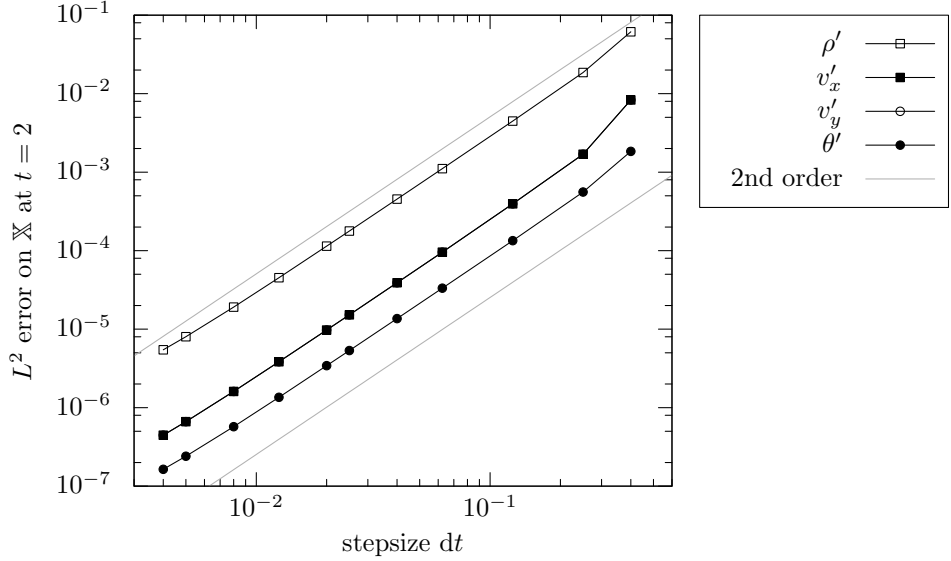}
\end{center}
\caption{$L^2$-error in the macroscopic density on the spatial domain $\mathbb{X} = [0;2)^2$ at time $t = 1$ for the diatomic D2Q5 velocity set and initial conditions (\ref{ic: 2D gauss}).}
\label{fig: L2-dens 2D gauss diatomic}
\end{figure}

\subsection{3D}

We calculate a 3D Gauss pulse for a monoatomic gas on the time-space domain $\mathbb{T} \times \mathbb{X}= [0;2] \times [0;2)^3$ for $\rho_0 = 1$ and $\theta_0 = \frac{1}{5}$ with initial conditions:
\begin{subequations}
\label{ic: 3D gauss}
\begin{align}
\rho'(t=0, x, y, z ) &= \exp\left( -15 \left|\left|(x,y,z)^T - (1,1,1)^T\right|\right|^2 \right), \\
u'_x(t=0, x, y, z ) &= 0,\\
u'_y( t=0, x, y, z ) &= 0,\\
u'_z( t=0, x, y, z ) &= 0,\\
\theta'(t=0, x, y, z ) &= 0.
\end{align}
\end{subequations}
Since the monoatomic case is of little interest in real-life situations, we analyze convergence of the method against a highly-resolved LBM solution in order to save compute resources.
From figure \ref{fig: L2-dens 3D gauss} we can see second-order convergence towards the numerical solution.
\begin{figure}[ht]
\begin{center}
\input{convergence_gauss_2D_L2_monoatomic}
\end{center}
\caption{$L^2$-error in the macroscopic variables for a monotaomic gas on the spatial domain $\mathbb{X} = [0;2)^3$ at time $t = 2$ for the D3Q7 velocity set and initial conditions (\ref{ic: 3D gauss}).}
\label{fig: L2-dens 3D gauss}
\end{figure}
\medskip
\\
For a diatomic gas, we first analyze the propagation of a Gauss pulse with dimensional background density $\ub{\rho}_0 = 1$ and dimensional background temperature $\ub{\theta}_0 = 1$ on the dimensional domain $\mathbb{\ub{T}} \times \mathbb{\ub{X}} = [0;1] \times [0;2)^3$.
The Gauss pulse is given by the dimensional initial condition:
\begin{subequations}
\label{ic: 3D diatomic}
\begin{align}
\ub{\rho}'(\ub{t}=0, \ub{x}, \ub{y}, \ub{z} ) &= \exp\left( -15 \left|\left|(\ub{x},\ub{y},\ub{z})^T - (1,1,1)^T\right|\right|^2 \right), \\
\ub{u'}_x(\ub{t}=0, \ub{x}, \ub{y}, \ub{z} ) &= 0,\\
\ub{u'}_y(\ub{t}=0, \ub{x}, \ub{y}, \ub{z} ) &= 0,\\
\ub{u'}_z(\ub{t}=0, \ub{x}, \ub{y}, \ub{z} ) &= 0,\\
\ub{\theta'}(\ub{t}=0, \ub{x}, \ub{y}, \ub{z} ) &= 0.
\end{align}
\end{subequations}
We use the diatomic D3Q7 velocity set for simulation.
The FVM solution was calculated using $1200 \times 1200 \times 1200$ volumes.
In order to compare the LBM and the FVM solutions, we need to use the non-dimensional end-time $T=\sqrt{\frac{21}{5}}$ for the LBM solver.
As we cannot find grid sizes for which the time step $\d t$ is a divisor of the dimensionless end-time $T=\sqrt{\frac{21}{5}}$, we need to take into account that we are comparing slightly different end-times for the FVM- and LBM-solutions.
The grid sizes were chosen in order to minimize this effect while still choosing the number of grid points in each direction $N$ such that it is a divisor of 1200.
Table \ref{tab: diff end-time minus LBM-end-time} presents the grid sizes used and the error between the dimensionless end-time $T=\sqrt{\frac{21}{5}}$ and the end-time of the according LBM simulation $N_{TS}\ \d t$ where $N_{TS} = \text{round}\left( \frac{T}{dt} \right)$.
From figure \ref{fig: L2-dens 3D gauss diatomic} we can see an second-order convergence for up to $N=80$.
The following decrease in convergence is due to the poor resolution of the FVM solution. 
\medskip
\\
In addition, we analyze convergence against a highly-resolved LBM solution for a Gauss pulse within a diatomic gas with initial conditions (\ref{ic: 3D gauss}) using the D3Q7 velocity.
In figure \ref{fig: L2-dens 3D gauss diatomic nondimensional} we see second-order convergence.

 \begin{table}[ht]
 \begin{center}
\begin{tabular}{|c|c|c|c|c|}
 \hline 
 $N$ & 25 & 40 & 80 & 120\\ 
 \hline 
 $\sqrt{\frac{21}{5}} - N_{TS}\ \d t$ & $-3.06\ 10^{-2}$ & $-6.10\ 10^{-4}$ & $-6.10\ 10^{-4}$ & $-6.10\ 10^{-4}$ \\ 
 \hline 
 \end{tabular}
 \end{center}
 \caption{Difference between the dimensionless end-time $T=\sqrt{\frac{21}{5}}$ and the end-time of the LBM simulations for different grid sizes in each direction $N$ with $N_{TS}$ time steps.}
 \label{tab: diff end-time minus LBM-end-time}
 \end{table}
 
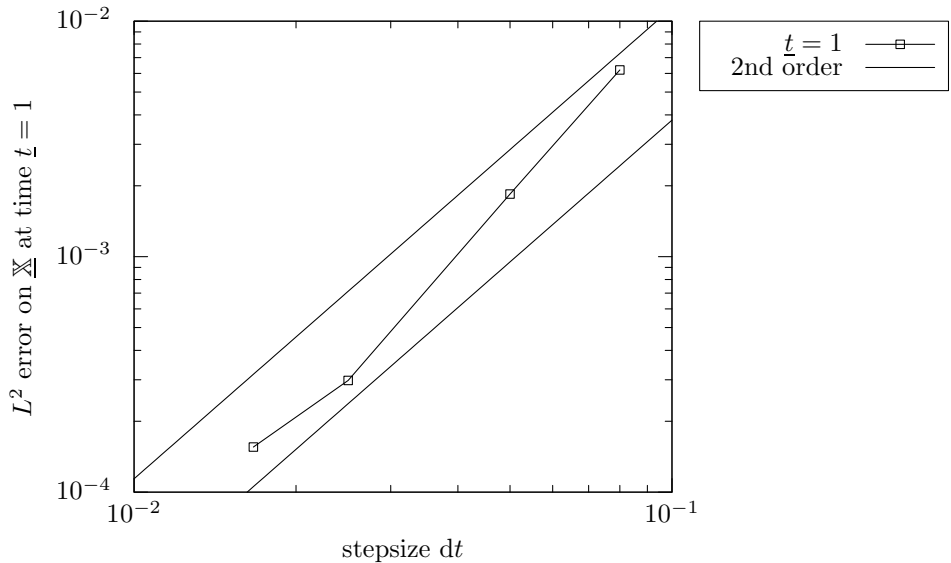
\begin{figure}[ht]
\begin{center}
\begin{tikzpicture}[gnuplot]
\gpsolidlines
\path (0.000,0.000) rectangle (12.700,7.620);
\gpcolor{color=gp lt color border}
\gpsetlinetype{gp lt border}
\gpsetlinewidth{1.00}
\draw[gp path] (1.688,0.985)--(1.868,0.985);
\draw[gp path] (8.839,0.985)--(8.659,0.985);
\node[gp node right] at (1.504,0.985) {$10^{-4}$};
\draw[gp path] (1.688,1.928)--(1.778,1.928);
\draw[gp path] (8.839,1.928)--(8.749,1.928);
\draw[gp path] (1.688,2.480)--(1.778,2.480);
\draw[gp path] (8.839,2.480)--(8.749,2.480);
\draw[gp path] (1.688,2.871)--(1.778,2.871);
\draw[gp path] (8.839,2.871)--(8.749,2.871);
\draw[gp path] (1.688,3.175)--(1.778,3.175);
\draw[gp path] (8.839,3.175)--(8.749,3.175);
\draw[gp path] (1.688,3.423)--(1.778,3.423);
\draw[gp path] (8.839,3.423)--(8.749,3.423);
\draw[gp path] (1.688,3.633)--(1.778,3.633);
\draw[gp path] (8.839,3.633)--(8.749,3.633);
\draw[gp path] (1.688,3.814)--(1.778,3.814);
\draw[gp path] (8.839,3.814)--(8.749,3.814);
\draw[gp path] (1.688,3.975)--(1.778,3.975);
\draw[gp path] (8.839,3.975)--(8.749,3.975);
\draw[gp path] (1.688,4.118)--(1.868,4.118);
\draw[gp path] (8.839,4.118)--(8.659,4.118);
\node[gp node right] at (1.504,4.118) {$10^{-3}$};
\draw[gp path] (1.688,5.061)--(1.778,5.061);
\draw[gp path] (8.839,5.061)--(8.749,5.061);
\draw[gp path] (1.688,5.613)--(1.778,5.613);
\draw[gp path] (8.839,5.613)--(8.749,5.613);
\draw[gp path] (1.688,6.004)--(1.778,6.004);
\draw[gp path] (8.839,6.004)--(8.749,6.004);
\draw[gp path] (1.688,6.308)--(1.778,6.308);
\draw[gp path] (8.839,6.308)--(8.749,6.308);
\draw[gp path] (1.688,6.556)--(1.778,6.556);
\draw[gp path] (8.839,6.556)--(8.749,6.556);
\draw[gp path] (1.688,6.766)--(1.778,6.766);
\draw[gp path] (8.839,6.766)--(8.749,6.766);
\draw[gp path] (1.688,6.947)--(1.778,6.947);
\draw[gp path] (8.839,6.947)--(8.749,6.947);
\draw[gp path] (1.688,7.108)--(1.778,7.108);
\draw[gp path] (8.839,7.108)--(8.749,7.108);
\draw[gp path] (1.688,7.251)--(1.868,7.251);
\draw[gp path] (8.839,7.251)--(8.659,7.251);
\node[gp node right] at (1.504,7.251) {$10^{-2}$};
\draw[gp path] (1.688,0.985)--(1.688,1.165);
\draw[gp path] (1.688,7.251)--(1.688,7.071);
\node[gp node center] at (1.688,0.677) {$10^{-2}$};
\draw[gp path] (3.841,0.985)--(3.841,1.075);
\draw[gp path] (3.841,7.251)--(3.841,7.161);
\draw[gp path] (5.100,0.985)--(5.100,1.075);
\draw[gp path] (5.100,7.251)--(5.100,7.161);
\draw[gp path] (5.993,0.985)--(5.993,1.075);
\draw[gp path] (5.993,7.251)--(5.993,7.161);
\draw[gp path] (6.686,0.985)--(6.686,1.075);
\draw[gp path] (6.686,7.251)--(6.686,7.161);
\draw[gp path] (7.253,0.985)--(7.253,1.075);
\draw[gp path] (7.253,7.251)--(7.253,7.161);
\draw[gp path] (7.731,0.985)--(7.731,1.075);
\draw[gp path] (7.731,7.251)--(7.731,7.161);
\draw[gp path] (8.146,0.985)--(8.146,1.075);
\draw[gp path] (8.146,7.251)--(8.146,7.161);
\draw[gp path] (8.512,0.985)--(8.512,1.075);
\draw[gp path] (8.512,7.251)--(8.512,7.161);
\draw[gp path] (8.839,0.985)--(8.839,1.165);
\draw[gp path] (8.839,7.251)--(8.839,7.071);
\node[gp node center] at (8.839,0.677) {$10^{-1}$};
\draw[gp path] (1.688,7.251)--(1.688,0.985)--(8.839,0.985)--(8.839,7.251)--cycle;
\node[gp node center,rotate=-270] at (0.246,4.118) {$L^2$ error on $\ub{\mathbb{X}}$ at time $\ub{t}=1$};
\node[gp node center] at (5.263,0.215) {stepsize $\d t$};
\draw[gp path] (9.207,6.327)--(9.207,7.251)--(12.515,7.251)--(12.515,6.327)--cycle;
\node[gp node right] at (11.231,6.943) {$\ub{t} = 1$};
\draw[gp path] (11.415,6.943)--(12.331,6.943);
\draw[gp path] (8.146,6.602)--(6.686,4.951)--(4.534,2.471)--(3.274,1.585);
\gpsetpointsize{4.00}
\gppoint{gp mark 4}{(8.146,6.602)}
\gppoint{gp mark 4}{(6.686,4.951)}
\gppoint{gp mark 4}{(4.534,2.471)}
\gppoint{gp mark 4}{(3.274,1.585)}
\gppoint{gp mark 4}{(11.873,6.943)}
\node[gp node right] at (11.231,6.635) {2nd order};
\draw[gp path] (11.415,6.635)--(12.331,6.635);
\draw[gp path] (1.688,1.163)--(1.760,1.226)--(1.832,1.290)--(1.905,1.353)--(1.977,1.416)%
  --(2.049,1.480)--(2.121,1.543)--(2.194,1.606)--(2.266,1.669)--(2.338,1.733)--(2.410,1.796)%
  --(2.483,1.859)--(2.555,1.923)--(2.627,1.986)--(2.699,2.049)--(2.771,2.113)--(2.844,2.176)%
  --(2.916,2.239)--(2.988,2.302)--(3.060,2.366)--(3.133,2.429)--(3.205,2.492)--(3.277,2.556)%
  --(3.349,2.619)--(3.422,2.682)--(3.494,2.745)--(3.566,2.809)--(3.638,2.872)--(3.711,2.935)%
  --(3.783,2.999)--(3.855,3.062)--(3.927,3.125)--(3.999,3.188)--(4.072,3.252)--(4.144,3.315)%
  --(4.216,3.378)--(4.288,3.442)--(4.361,3.505)--(4.433,3.568)--(4.505,3.632)--(4.577,3.695)%
  --(4.650,3.758)--(4.722,3.821)--(4.794,3.885)--(4.866,3.948)--(4.938,4.011)--(5.011,4.075)%
  --(5.083,4.138)--(5.155,4.201)--(5.227,4.264)--(5.300,4.328)--(5.372,4.391)--(5.444,4.454)%
  --(5.516,4.518)--(5.589,4.581)--(5.661,4.644)--(5.733,4.708)--(5.805,4.771)--(5.877,4.834)%
  --(5.950,4.897)--(6.022,4.961)--(6.094,5.024)--(6.166,5.087)--(6.239,5.151)--(6.311,5.214)%
  --(6.383,5.277)--(6.455,5.340)--(6.528,5.404)--(6.600,5.467)--(6.672,5.530)--(6.744,5.594)%
  --(6.816,5.657)--(6.889,5.720)--(6.961,5.784)--(7.033,5.847)--(7.105,5.910)--(7.178,5.973)%
  --(7.250,6.037)--(7.322,6.100)--(7.394,6.163)--(7.467,6.227)--(7.539,6.290)--(7.611,6.353)%
  --(7.683,6.416)--(7.756,6.480)--(7.828,6.543)--(7.900,6.606)--(7.972,6.670)--(8.044,6.733)%
  --(8.117,6.796)--(8.189,6.859)--(8.261,6.923)--(8.333,6.986)--(8.406,7.049)--(8.478,7.113)%
  --(8.550,7.176)--(8.622,7.239)--(8.636,7.251);
\draw[gp path] (3.191,0.985)--(3.205,0.997)--(3.277,1.061)--(3.349,1.124)--(3.422,1.187)%
  --(3.494,1.251)--(3.566,1.314)--(3.638,1.377)--(3.711,1.441)--(3.783,1.504)--(3.855,1.567)%
  --(3.927,1.630)--(3.999,1.694)--(4.072,1.757)--(4.144,1.820)--(4.216,1.884)--(4.288,1.947)%
  --(4.361,2.010)--(4.433,2.073)--(4.505,2.137)--(4.577,2.200)--(4.650,2.263)--(4.722,2.327)%
  --(4.794,2.390)--(4.866,2.453)--(4.938,2.516)--(5.011,2.580)--(5.083,2.643)--(5.155,2.706)%
  --(5.227,2.770)--(5.300,2.833)--(5.372,2.896)--(5.444,2.960)--(5.516,3.023)--(5.589,3.086)%
  --(5.661,3.149)--(5.733,3.213)--(5.805,3.276)--(5.877,3.339)--(5.950,3.403)--(6.022,3.466)%
  --(6.094,3.529)--(6.166,3.592)--(6.239,3.656)--(6.311,3.719)--(6.383,3.782)--(6.455,3.846)%
  --(6.528,3.909)--(6.600,3.972)--(6.672,4.036)--(6.744,4.099)--(6.816,4.162)--(6.889,4.225)%
  --(6.961,4.289)--(7.033,4.352)--(7.105,4.415)--(7.178,4.479)--(7.250,4.542)--(7.322,4.605)%
  --(7.394,4.668)--(7.467,4.732)--(7.539,4.795)--(7.611,4.858)--(7.683,4.922)--(7.756,4.985)%
  --(7.828,5.048)--(7.900,5.111)--(7.972,5.175)--(8.044,5.238)--(8.117,5.301)--(8.189,5.365)%
  --(8.261,5.428)--(8.333,5.491)--(8.406,5.555)--(8.478,5.618)--(8.550,5.681)--(8.622,5.744)%
  --(8.695,5.808)--(8.767,5.871)--(8.839,5.934);
\draw[gp path] (1.688,7.251)--(1.688,0.985)--(8.839,0.985)--(8.839,7.251)--cycle;
\gpdefrectangularnode{gp plot 1}{\pgfpoint{1.688cm}{0.985cm}}{\pgfpoint{8.839cm}{7.251cm}}
\end{tikzpicture}
\end{center}
\caption{$L^2$-error in the density on the domain $\ub{\mathbb{X}}=[0,2]^3$ at time $\ub{t}=1$ for the D3Q7 velocity set and initial conditions (\ref{ic: 3D diatomic}).}
\label{fig: L2-dens 3D gauss diatomic}
\end{figure}

\begin{figure}[ht]
\begin{center}
\input{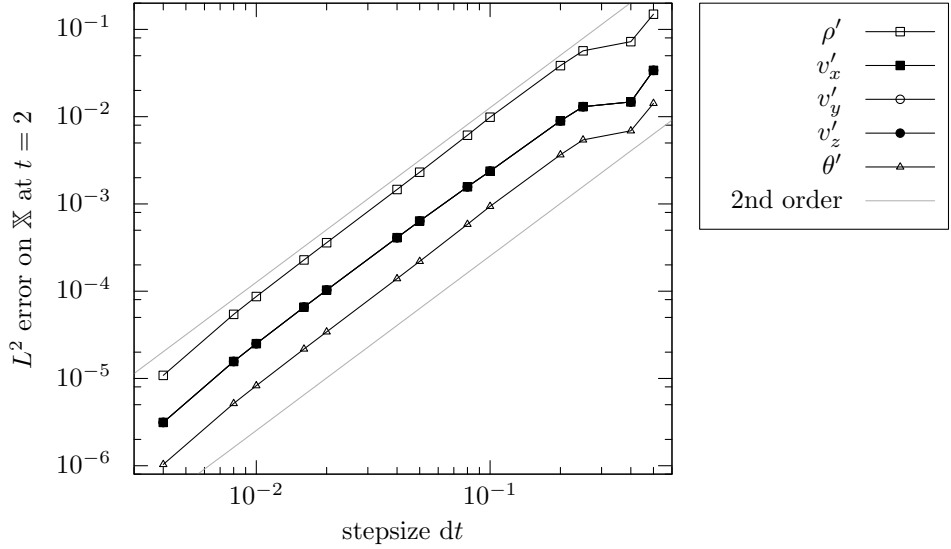}
\end{center}
\caption{$L^2$-error in the macroscopic variables for a diatomic gas on the domain $\mathbb{X}=[0,2]^3$ at time $t=2$ for the D3Q7 velocity set and initial conditions (\ref{ic: 3D gauss}).}
\label{fig: L2-dens 3D gauss diatomic nondimensional}
\end{figure}
 
 \section{Conclusions}
 \label{sec: conclusions}

In this work we derived second-order convergent Lattice Boltzmann schemes solving the Linearized Euler Equations without background velocity for arbitrary adiabatic exponents $\gamma$.
We were able to state conditions on how to choose the velocity sets in the Finite Discrete Velocity Model and the Lattice Boltzmann Method and for stability of the Lattice Boltzmann Method.
Due to our generic approach, the results from this work can be used as a blueprint for the derivation of Lattice Boltzmann Methods for different lattice geometries or adiabatic exponents.
In our future work we will focus on integration of background velocities into these Lattice Boltzmann Methods, analysis of boundary conditions, and examination of coupling strategies to Lattice Boltzmann Methods solving the incompressible Navier Stokes Equations and Finite Volume Methods solving the Linearized Euler Equations.

\section*{Acknowledgments}

We would like to thank our reviewer for pointing out the work of Banda \etal \cite{banda2006}, Junk and Yong \cite{JunkYong2009}, Junk and Yang \cite{JunkYang2009}, and Yong \cite{Yong2009862}.
This approach towards stability was previously unknown to us and we expect it to be helpful in the derivation of Lattice Boltzmann schemes for the Linearized Euler Equations with background velocity.
In addition, we would like to thank Wen-An Yong for fruitful discussion about his work on stability of hyperbolic relaxation systems.

\begin{appendix}

\section{Analysis of the LBE}
\label{app: analysis of the LBE}

We carry out the steps of deriving equations (\ref{eq: eps to 2 zeroth})-(\ref{eq: eps to 2 second}) starting from the moments of equation (\ref{eq: taylor eps to 2}) for $\zeta \in \left\{\vec{v}\to 1, \vec{v}\to v_1,..., \vec{v}\to v_D, \vec{v}\to \frac{1}{2} |\vec{v}|^2\right\}$:
$$
\nabla_{\vec{x}} \cdot \langle  \zeta \vec{v} g^{(1)} \rangle_\mathcal{S}   + \frac{1}{2} \nabla_{\vec{x}} \nabla_{\vec{x}} : \langle  \zeta \vec{v} \otimes \vec{v} g^{(0)} \rangle_\mathcal{S}   + \partial_t \langle  \zeta g^{(1)} \rangle_\mathcal{S}   + \partial_t \nabla_{\vec{x}} \cdot \langle  \zeta \vec{v} g^{(0)} \rangle_\mathcal{S}   + \frac{1}{2} \partial_t^2 \langle  \zeta g^{(0)} \rangle_\mathcal{S}   = 0
$$
We start with the case $\zeta = 1$:
\begin{align*}
0 =&\nabla_{\vec{x}} \cdot \langle  \vec{v} g^{(1)} \rangle_\mathcal{S}   + \frac{1}{2} \nabla_{\vec{x}} \nabla_{\vec{x}} : \langle  \vec{v} \otimes \vec{v} g^{(0)} \rangle_\mathcal{S}\\
&+ \partial_t\langle g^{(1)} \rangle_\mathcal{S}   + \partial_t \nabla_{\vec{x}} \cdot \langle  \vec{v} g^{(0)} \rangle_\mathcal{S}   + \frac{1}{2} \partial_t^2 \langle  g^{(0)} \rangle_\mathcal{S}  \\
=& \partial_t \rho'^{(1)} + \rho_0 \nabla_{\vec{x}} \cdot \vec{u}'^{(1)} + \frac{1}{2} \nabla_{\vec{x}} \cdot \underbrace{\left( \rho_0 \partial_t \vec{u}'^{(0)} + \theta_0 \nabla_{\vec{x}} \rho'^{(0)} + \rho_0 \nabla_{\vec{x}} \theta'^{(0)} \right)}_{=0}\\
& + \frac{1}{2} \partial_t \underbrace{\left( \partial_t \rho'^{(0)} + \rho_0 \nabla_{\vec{x}} \cdot \vec{u}'^{(0)} \right)}_{=0 }\\
=& \partial_t \rho'^{(1)} + \rho_0 \nabla_{\vec{x}} \cdot \vec{u}'^{(1)},
\end{align*}
due to the term $\langle \vec{v} \otimes \vec{v} g^{(0)}\rangle_\mathcal{S}  $ being of the following form:
\begin{align*}
\langle \vec{v} \otimes \vec{v} g^{(0)}\rangle_\mathcal{S}   &= \left( \theta_0 \rho'^{(0)} + \rho_0 \theta'^{(0)} \right) \vec{I}.
\end{align*}
We proceed with the case $\zeta = \vec{v}$:
\begin{align}
\vec{0} =& \nabla_{\vec{x}} \cdot \langle  \vec{v} \otimes \vec{v} g^{(1)} \rangle_\mathcal{S}   + \frac{1}{2} \nabla_{\vec{x}} \nabla_{\vec{x}} : \langle  \vec{v} \otimes \vec{v} \otimes \vec{v} g^{(0)} \rangle_\mathcal{S}   + \partial_t\langle \vec{v} g^{(1)} \rangle_\mathcal{S}\nonumber\\
& + \partial_t \nabla_{\vec{x}} \cdot \langle  \vec{v} \otimes \vec{v} g^{(0)} \rangle_\mathcal{S}   + \frac{1}{2} \partial_t^2 \langle  \vec{v} g^{(0)} \rangle_\mathcal{S}\nonumber\\
=& \rho_0 \partial_t \vec{u}'^{(1)} + \nabla_{\vec{x}} \cdot \langle \vec{v} \otimes \vec{v} g^{(1)} \rangle_\mathcal{S}\nonumber\\
&+ \frac{1}{2} \nabla_{\vec{x}} \cdot \left( \rho_0 \partial_t \theta'^{(0)} \vec{I} + \theta_0 \partial_t \rho'^{(0)} \vec{I} + \nabla_{\vec{x}} \cdot \langle \vec{v} \otimes \vec{v} \otimes \vec{v} g^{(0)} \rangle_\mathcal{S} \right)\nonumber\\
&+ \frac{1}{2} \partial_t \underbrace{\left( \rho_0 \partial_t \vec{u}'^{(0)} + \rho_0 \nabla_{\vec{x}} \theta'^{(0)} + \theta_0 \nabla_{\vec{x}} \rho'^{(0)} \right)}_{=0}. \label{eq: first order cont}
\end{align}
We now analyze the term $\langle  \vec{v} \otimes \vec{v} g^{(1)} \rangle_\mathcal{S}$.
From equation (\ref{eq: taylor eps to 1}) we find:
$$
g_i^{(1)} = g_i^{eq(1)} - \tau \left( \vec{c}_i \cdot \nabla_{\vec{x}} g_i^{(0)} + \partial_t g_i^{(0)} \right).
$$
Therefore we have:
\begin{align*}
\langle  \vec{v} \otimes \vec{v} g^{(1)} \rangle_\mathcal{S}   &= \langle  \vec{v} \otimes \vec{v} g^{eq(1)} \rangle_\mathcal{S}   - \tau \left( \nabla_{\vec{x}} \cdot \langle  \vec{v} \otimes \vec{v} \otimes \vec{v} g^{eq(0)}\rangle_\mathcal{S}   + \partial_t \langle  \vec{v} \otimes \vec{v} g^{eq(0)}\rangle_\mathcal{S}   \right)\\
=& \left( \theta_0 \rho'^{(1)} + \rho_0 \theta'^{(1)} \right) \vec{I}\nonumber\\
&- \tau \left( \nabla_{\vec{x}} \cdot \langle  \vec{v} \otimes \vec{v} \otimes \vec{v} g^{(0)} \rangle_\mathcal{S}  + \partial_t \left( \theta_0 \rho'^{(0)} + \rho_0 \theta'^{(0)} \right) \vec{I} \right).
\end{align*}
Now we plug these results into equation (\ref{eq: first order cont}) in order to calculate the first order moment of equation (\ref{eq: eps to 2 second}):
\begin{align*}
\vec{0}=& \rho_0 \partial_t \vec{u}'^{(1)} + \rho_0 \nabla_{\vec{x}} \theta'^{(1)} + \theta_0 \nabla_{\vec{x}} \rho'^{(1)}\\
+& \left( \frac{1}{2} - \tau \right)  \nabla_{\vec{x}} \nabla_{\vec{x}} : \langle  \vec{v} \otimes \vec{v} \otimes \vec{v} g^{(0)} \rangle_\mathcal{S}\\
+&\left( \frac{1}{2} - \tau \right) \partial_t \nabla_{\vec{x}} \left( \rho_0 \theta'^{(0)} + \theta_0 \rho'^{(0)} \right).
\end{align*}
Finally, we proceed with the case $\zeta = \frac{1}{2} \left( |\vec{v}|^2 + \beta \right)$:
\begin{align}
0 &=\nabla_{\vec{x}} \cdot \langle  \frac{1}{2} \left( |\vec{v}|^2 + \beta \right) \vec{v} g^{(1)} \rangle_\mathcal{S}   + \frac{1}{2} \nabla_{\vec{x}} \nabla_{\vec{x}} : \langle  \frac{1}{2} \left( |\vec{v}|^2 + \beta \right) \vec{v} \otimes \vec{v} g^{(0)} \rangle_\mathcal{S}\nonumber\\
&+\partial_t\langle \frac{1}{2} \left( |\vec{v}|^2 + \beta \right) g^{(1)} \rangle_\mathcal{S}   + \partial_t \nabla_{\vec{x}} \cdot \langle  \frac{1}{2} \left( |\vec{v}|^2 + \beta \right) \vec{v} g^{(0)} \rangle_\mathcal{S}   + \frac{1}{2} \partial_t^2 \langle  \frac{1}{2} \left( |\vec{v}|^2 + \beta \right) g^{(0)} \rangle_\mathcal{S}\nonumber \\
&= \frac{1}{\gamma-1} \rho_0 \partial_t \theta'^{(1)} + \frac{1}{\gamma-1} \theta_0 \partial_t \rho'^{(1)} + \nabla_{\vec{x}} \cdot \langle \frac{1}{2} \left( |\vec{v}|^2 + \beta \right) \vec{v} g^{(1)} \rangle_\mathcal{S}\nonumber\\
&+ \frac{1}{2} \nabla_{\vec{x}} \cdot \left( \frac{\gamma}{\gamma-1} \rho_0 \theta_0 \partial_t \vec{u}'^{(0)} + \nabla_{\vec{x}} \cdot \langle \frac{1}{2} \left( |\vec{v}|^2 + \beta \right) \vec{v} \otimes \vec{v} g^{(0)} \rangle_\mathcal{S} \right)\nonumber\\
&+ \frac{1}{2} \partial_t \underbrace{\left( \frac{1}{\gamma-1} \rho_0 \partial_t \theta'^{(0)} + \frac{1}{\gamma-1} \theta_0 \partial_t \rho'^{(0)} + \frac{\gamma}{\gamma-1} \rho_0 \theta_0 \nabla_{\vec{x}} \vec{u}'^{(0)} \right)}_{=0}. \label{eq: second order cont}
\end{align}
Analysis of the terms $\langle \frac{1}{2} \left( |\vec{v}|^2 + \beta \right) g^{(1)} \rangle_\mathcal{S}$ and $\langle  \frac{1}{2} \left( |\vec{v}|^2 + \beta \right) \vec{v} g^{(1)} \rangle_\mathcal{S}  $ gives:
\begin{align*}
\langle \frac{1}{2} \left( |\vec{v}|^2 + \beta \right) g^{(1)} \rangle_\mathcal{S} =& \langle \frac{1}{2} \left( |\vec{v}|^2 + \beta \right) g^{eq(1)} \rangle_\mathcal{S}\nonumber\\
&- \tau \left( \nabla_{\vec{x}} \cdot \langle \frac{1}{2} \left( |\vec{v}|^2 + \beta \right) \vec{v} g^{(0)} \rangle_\mathcal{S} + \partial_t \langle \frac{1}{2} \left( |\vec{v}|^2 + \beta \right) g^{(0)} \rangle_\mathcal{S} \right)\\
=& \frac{1}{\gamma - 1} \rho_0 \theta'^{(1)} + \frac{1}{\gamma - 1} \theta_0 \rho'^{(1)}\nonumber\\
&- \tau \underbrace{\left(\frac{1}{\gamma - 1} \rho_0 \partial_t \theta'^{(0)} + \frac{1}{\gamma - 1} \theta_0 \partial_t \rho'^{(0)}  + \frac{\gamma}{\gamma-1} \nabla_{\vec{x}} \cdot  u'^{(0)} \right)}_{=0},\\ 
\langle \frac{1}{2} \left( |\vec{v}|^2 + \beta \right) \vec{v} g^{(1)} \rangle_\mathcal{S} =& \langle \frac{1}{2} \left( |\vec{v}|^2 + \beta \right) \vec{v} g^{eq(1)} \rangle_\mathcal{S}\nonumber\\
&- \tau \left( \nabla_{\vec{x}} \cdot \langle \frac{1}{2} \left( |\vec{v}|^2 + \beta \right) \vec{v} \otimes \vec{v} g^{(0)} \rangle_\mathcal{S} + \partial_t \langle \frac{1}{2} \left( |\vec{v}|^2 + \beta \right) \vec{v} g^{(0)} \rangle_\mathcal{S} \right)\\
=& \frac{\gamma}{\gamma-1} \rho_0 \theta_0 \vec{u}'^{(1)}\nonumber\\
&- \tau \left( \nabla_{\vec{x}} \cdot \langle \frac{1}{2} \left( |\vec{v}|^2 + \beta \right) \vec{v} \otimes \vec{v} g^{(0)} \rangle_\mathcal{S} + \frac{\gamma}{\gamma-1} \rho_0 \theta_0 \partial_t \vec{u}'^{(0)} \right).
\end{align*}
Plugging this into equation (\ref{eq: second order cont}) gives the central second-order moment of equation (\ref{eq: eps to 2 second}):
\begin{align*}
0 =& \frac{1}{\gamma-1} \rho_0 \partial_t \theta'^{(1)} + \frac{1}{\gamma-1} \theta_0 \partial_t \rho'^{(1)} + \frac{\gamma}{\gamma-1} \rho_0 \theta_0 \nabla_{\vec{x}} \cdot \vec{u}'^{(1)}\\
+& \left( \frac{1}{2} - \tau \right) \frac{\gamma}{\gamma-1} \rho_0 \theta_0 \partial_t \nabla_{\vec{x}} \cdot \vec{u}'^{(0)} + \left( \frac{1}{2} - \tau \right) \nabla_{\vec{x}} \nabla_{\vec{x}} : \langle \frac{1}{2} \left( |\vec{v}|^2 + \beta \right) \vec{v} \otimes \vec{v} g^{(0)} \rangle_\mathcal{S}
\end{align*}

\end{appendix}

\section*{Literature}

\bibliographystyle{abbrv}
\bibliography{LBM}

\begin{thebibliography}{10}

\bibitem{bailly2000numerical}
C.~Bailly and D.~Juv{\'e}.
\newblock Numerical solution of acoustic propagation problems using linearized
  euler equations.
\newblock {\em AIAA journal}, 38(1):22--29, 2000.

\bibitem{petsc-user-ref}
S.~Balay, M.~F. Adams, J.~Brown, P.~Brune, K.~Buschelman, V.~Eijkhout, W.~D.
  Gropp, D.~Kaushik, M.~G. Knepley, L.~C. McInnes, K.~Rupp, B.~F. Smith, and
  H.~Zhang.
\newblock {PETS}c users manual.
\newblock Technical Report ANL-95/11 - Revision 3.4, Argonne National
  Laboratory, 2013.

\bibitem{petsc-web-page}
S.~Balay, M.~F. Adams, J.~Brown, P.~Brune, K.~Buschelman, V.~Eijkhout, W.~D.
  Gropp, D.~Kaushik, M.~G. Knepley, L.~C. McInnes, K.~Rupp, B.~F. Smith, and
  H.~Zhang.
\newblock {PETS}c web page.
\newblock \url{http://www.mcs.anl.gov/petsc}, 2014.

\bibitem{petsc-efficient}
S.~Balay, W.~D. Gropp, L.~C. McInnes, and B.~F. Smith.
\newblock Efficient management of parallelism in object oriented numerical
  software libraries.
\newblock In E.~Arge, A.~M. Bruaset, and H.~P. Langtangen, editors, {\em Modern
  Software Tools in Scientific Computing}, pages 163--202. Birkh{\"{a}}user
  Press, 1997.

\bibitem{banda2006}
M.~K. Banda, W.-A. Yong, and A.~Klar.
\newblock A stability notion for lattice boltzmann equations.
\newblock {\em SIAM Journal on Scientific Computing}, 27(6):2098--2111, 2006.

\bibitem{Bardos2000}
C.~Bardos, F.~Golse, and C.~D. Levermore.
\newblock The acoustic limit for the boltzmann equation.
\newblock {\em Archive for Rational Mechanics and Analysis}, 153(3):177--204,
  June 2000.

\bibitem{bauer1960}
F.~L. Bauer and C.~T. Fike.
\newblock Norms and exclusion theorems.
\newblock {\em Numerische Mathematik}, 2(1):137--141, 1960.

\bibitem{bernsdorf1999comparison}
J.~Bernsdorf, F.~Durst, and M.~Sch{\"a}fer.
\newblock Comparison of cellular automata and finite volume techniques for
  simulation of incompressible flows in complex geometries.
\newblock {\em International Journal for Numerical Methods in Fluids},
  29(3):251--264, 1999.

\bibitem{BGK1945}
P.~L. Bhatnagar, E.~P. Gross, and M.~Krook.
\newblock A model for collision processes in gases. i. small amplitude
  processes in charged and neutral one-component systems.
\newblock {\em Physical review}, 94(3):511, 1954.

\bibitem{bogey2002computation}
C.~Bogey, C.~Bailly, and D.~Juv{\'e}.
\newblock Computation of flow noise using source terms in linearized euler's
  equations.
\newblock {\em AIAA journal}, 40(2):235--243, 2002.

\bibitem{buick1998lattice}
J.~Buick, C.~Greated, and D.~Campbell.
\newblock Lattice {BGK} simulation of sound waves.
\newblock {\em Europhysics Letters}, 43(3):235, 1998.

\bibitem{Cercignani1988}
C.~Cercignani.
\newblock {\em The Boltzmann Equation and Its Applications}.
\newblock Applied Mathematical Sciences Series. Springer, 1988.

\bibitem{HTheoremInstability}
H.~Chen and C.~Teixeira.
\newblock H-theorem and origins of instability in thermal lattice boltzmann
  models.
\newblock {\em Computer Physics Communications}, 129(1--3):21 -- 31, 2000.

\bibitem{clawpack}
{Clawpack Development Team}.
\newblock Clawpack software, 2013.
\newblock Version 5.0.

\bibitem{crouse2006fundamental}
B.~Crouse, D.~Freed, G.~Balasubramanian, S.~Senthooran, P.-T. Lew, and
  L.~Mongeau.
\newblock {\em Fundamental Aeroacoustics Capabilities of the Lattice-Boltzmann
  Method}.
\newblock American Institute of Aeronautics and Astronautics, 2006.

\bibitem{DegondJin2005}
P.~Degond and S.~Jin.
\newblock A smooth transition model between kinetic and diffusion equations.
\newblock {\em SIAM Journal on Numerical Analysis}, 42(6):pp. 2671--2687, 2005.

\bibitem{Degond2005}
P.~Degond, S.~Jin, and L.~Mieussens.
\newblock {A smooth transition model between kinetic and hydrodynamic
  equations}.
\newblock {\em Journal of Computational Physics}, 209(2):665--694, 2005.

\bibitem{dellar2001bulk}
P.~J. Dellar.
\newblock Bulk and shear viscosities in lattice boltzmann equations.
\newblock {\em Physical Review E}, 64(3):031203, August 2001.

\bibitem{dellar2008two}
P.~J. Dellar.
\newblock Two routes from the boltzmann equation to compressible flow of
  polyatomic gases.
\newblock {\em Progress in Computational Fluid Dynamics, an International
  Journal}, 8(1):84--96, 2008.

\bibitem{DissManuel}
M.~Hasert.
\newblock {\em Multi-scale Lattice Boltzmann Simulations on Distributed
  Octrees}.
\newblock PhD thesis, RWTH Aachen University, 2014.

\bibitem{hasert2011towards}
M.~Hasert, J.~Bernsdorf, and S.~Roller.
\newblock Towards aeroacoustic sound generation by flow through porous media.
\newblock {\em Philosophical Transactions of the Royal Society A: Mathematical,
  Physical and Engineering Sciences}, 369(1945):2467--2475, 2011.

\bibitem{HeLuo1997}
X.~He and L.-S. Luo.
\newblock \textit{A priori} derivation of the lattice boltzmann equation.
\newblock {\em Physical Review E}, 55(6):R6333--R6336, 1997.

\bibitem{FELBM}
J.~C. Jo, K.~W. Rhoh, and Y.~W. Kwon.
\newblock Finite element based formulation of the lattice boltzmann equation.
\newblock {\em Nuclear Engineering and Technology}, 41(5):649--654, 2009.

\bibitem{Junk2005}
M.~Junk, A.~Klar, and L.-S. Luo.
\newblock {Asymptotic analysis of the lattice Boltzmann equation}.
\newblock {\em Journal of Computational Physics}, 210(2):676--704, Dec. 2005.

\bibitem{JunkYang2009}
M.~Junk and Z.~Yang.
\newblock Convergence of lattice boltzmann methods for navier--stokes flows in
  periodic and bounded domains.
\newblock {\em Numerische Mathematik}, 112(1):65--87, 2009.

\bibitem{JunkYong2009}
M.~Junk and W.-A. Yong.
\newblock Weighted {$L^{2}$}-stability of the lattice boltzmann method.
\newblock {\em SIAM Journal on Numerical Analysis}, 47(3):1651--1665, 2009.

\bibitem{EntropyFunctionsLBM}
I.~V. Karlin, A.~Ferrante, and H.~C. {\"O}ttinger.
\newblock Perfect entropy functions of the lattice boltzmann method.
\newblock {\em EPL (Europhysics Letters)}, 47(2):182--188, 1999.

\bibitem{kataoka2004}
T.~Kataoka and M.~Tsutahara.
\newblock Lattice boltzmann model for the compressible navier-stokes equations
  with flexible specific-heat ratio.
\newblock {\em Physical review E}, 69(3):035701, 2004.

\bibitem{pyclaw-sisc}
D.~I. Ketcheson, K.~T. Mandli, A.~J. Ahmadia, A.~Alghamdi, M.~{Quezada de
  Luna}, M.~Parsani, M.~G. Knepley, and M.~Emmett.
\newblock {P}y{C}law: Accessible, extensible, scalable tools for wave
  propagation problems.
\newblock {\em SIAM Journal on Scientific Computing}, 34(4):C210--C231, Nov.
  2012.

\bibitem{Lallemand2000a}
P.~Lallemand and L.~Luo.
\newblock {Theory of the lattice boltzmann method: dispersion, dissipation,
  isotropy, galilean invariance, and stability}.
\newblock {\em Phys. Rev. E}, 61(6 Pt A):6546--62, June 2000.

\bibitem{Lallemand2003}
P.~Lallemand and L.-S. Luo.
\newblock Theory of the lattice boltzmann method: Acoustic and thermal
  properties in two and three dimensions.
\newblock {\em Physical Review E}, 68(3):036706, Sept. 2003.

\bibitem{LangsethLeVeque00}
J.~O. Langseth and R.~J. LeVeque.
\newblock A wave-propagation method for three-dimensional hyperbolic
  conservation laws.
\newblock {\em J. Comput. Phys.}, 165:126--166, 2000.

\bibitem{CGLBM}
T.~Lee and C.-L. Lin.
\newblock A characteristic galerkin method for discrete boltzmann equation.
\newblock {\em Journal of Computational Physics}, 171(1):336--356, 2001.

\bibitem{leveque1997}
R.~J. LeVeque.
\newblock Wave propagation algorithms for multidimensional hyperbolic systems.
\newblock {\em Journal of Computational Physics}, 131(2):327--353, 1997.

\bibitem{LeVeque2007}
R.~J. LeVeque.
\newblock {\em Finite Difference Methods for Ordinary and Partial Differential
  Equations}.
\newblock Society for Industrial and Applied Mathematics, 2007.

\bibitem{pyclaw}
K.~T. Mandli and D.~I.~e. Ketcheson.
\newblock Py{C}law software, 2011.
\newblock Version 5.0.

\bibitem{mankbadi1998use}
R.~Mankbadi, R.~Hixon, S.~Shih, and L.~Povinelli.
\newblock Use of linearized euler equations for supersonic jet noise
  prediction.
\newblock {\em AIAA journal}, 36(2):140--147, 1998.

\bibitem{Marie2009}
S.~Marie, D.~Ricot, and P.~Sagaut.
\newblock {Comparison between lattice Boltzmann method and Navier--Stokes high
  order schemes for computational aeroacoustics}.
\newblock {\em Journal of Computational Physics}, 228(4):1056--1070, Mar. 2009.

\bibitem{roller2005}
S.~Roller, T.~Schwartzkopff, R.~Fortenbach, M.~Dumbser, and C.-D. Munz.
\newblock Calculation of low mach number acoustics: a comparison of {MPV},
  {EIF} and linearized euler equations.
\newblock {\em ESAIM: Mathematical Modelling and Numerical Analysis},
  39(03):561--576, 2005.

\bibitem{SterlingChen96}
J.~D. Sterling and S.~Chen.
\newblock Stability analysis of lattice boltzmann methods.
\newblock {\em Journal of Computational Physics}, 123(1):196--206, 1996.

\bibitem{SucciBook}
S.~Succi.
\newblock {\em {The Lattice Boltzmann Equation for Fluid Dynamics and Beyond
  (Numerical Mathematics and Scientific Computation)}}.
\newblock Numerical mathematics and scientific computation. Oxford University
  Press, USA, Aug. 2001.

\bibitem{toro1999riemann}
E.~F. Toro.
\newblock {\em Riemann Solvers and Numerical Methods for Fluid Dynamics},
  volume~16.
\newblock Springer, 1999.

\bibitem{Trefethen1996}
L.~N. Trefethen.
\newblock \textit{Finite Difference and Spectral Methods for Ordinary and
  Partial Differential Equations}.
\newblock http://people.maths.ox.ac.uk/trefethen/4all.pdf, 1996.

\bibitem{Yong2009862}
W.-A. Yong.
\newblock An onsager-like relation for the lattice boltzmann method.
\newblock {\em Computers \& Mathematics with Applications}, 58(5):862 -- 866,
  2009.

\end{thebibliography}

\end{document}